\documentclass[11pt,reqno]{amsart}

\usepackage{amssymb}
\usepackage{amsmath}
\usepackage{amsthm}
\usepackage{latexsym}
\usepackage{amsfonts}
\usepackage{mathrsfs}
\usepackage{url}
\usepackage{titlesec}
\usepackage{color}

\newtheorem{thm}{Theorem}[section]
\newtheorem{cor}[thm]{Corollary}
\newtheorem{lem}[thm]{Lemma}
\newtheorem{prop}[thm]{Proposition}

\newtheorem{defn}[thm]{Definition}
\newtheorem{rem}[thm]{Remark}
\newtheorem{ex}[thm]{Example}

\numberwithin{equation}{section}
\titleformat{\section}{\normalfont\bfseries\centering}{\thesection.}{.25em}{}
\titleformat{\subsection}{\normalfont\bfseries}{\thesubsection.}{.25em}{}
\newcommand{\paragraf}{\textsection}

\newcommand{\R}{\ensuremath{\mathbb R}}    % Reelle Zahlen
\newcommand{\C}{\ensuremath{\mathbb C}}    % Komplexe Zahlen
\newcommand{\N}{\ensuremath{\mathbb N}}    % Nat"urliche Zahlen
\newcommand{\<}{\langle}
\renewcommand{\>}{\rangle}

\newcommand{\calA}{\mathcal A}         
\newcommand{\calB}{\mathcal B}         \newcommand{\frakB}{\mathfrak B}
\newcommand{\calC}{\mathcal C}         
\newcommand{\calD}{\mathcal D}         
         \newcommand{\frakE}{\mathfrak E}

\newcommand{\calH}{\mathcal H}         \newcommand{\frakH}{\mathfrak H}

\newcommand{\calK}{\mathcal K}         \newcommand{\frakK}{\mathfrak K}
\newcommand{\calL}{\mathcal L}

\newcommand{\calQ}{\mathcal Q}         
\newcommand{\calR}{\mathcal R}

\newcommand{\calV}{\mathcal V}         
\newcommand{\calW}{\mathcal W}

\newcommand{\la}{\lambda}
\newcommand{\veps}{\varepsilon}
\newcommand{\vphi}{\varphi}

\newcommand{\linspan}{\operatorname{span}}
\renewcommand{\ker}{\operatorname{ker}}
\newcommand{\ran}{\operatorname{ran}}

\newcommand{\Lra}{\Longrightarrow}
\newcommand{\Sra}{\Rightarrow}

\newcommand{\Llra}{\Longleftrightarrow}

\newcommand{\ol}{\overline}

\newcommand{\wt}{\widetilde}

\definecolor{darkgreen}{rgb}{0,0.6,0.1}

\newcommand{\tr}{\operatorname{Tr}}
\newcommand{\cls}{\ol{\linspan}\,}%{\operatorname{c.l.s.}}
\newcommand{\dist}{\operatorname{dist}}
\newcommand{\rmref}[1]{{\rm\ref{#1}}}

\newcommand{\diag}{\mathrm{diag}}

\newcommand{\braces}[1]{{\rm (}#1{\rm )}}
\begin{document}
%%%%%%%%%%%%%%%%%%%%%%%%%%%%%%%%%%%%%%%%%%%%%%%%%%%%%%%%%%%%%%%%%%%%%%%%%%%%%
%%%
%%%  HEAD OF PAPER
%%%
\title[The Effect of Perturbations on Duals]{The Effect of Perturbations of Operator-Valued Frame Sequences and Fusion Frames on Their Duals}

%Perturbations of Operator-valued Frame Sequences and Fusion Frames and the Effect on their Duals}

\author{Gitta Kutyniok}
\address{Institut f\"ur Mathematik, Technische Universit\"at Berlin, Stra\ss e des 17.\ Juni 136, 10623 Berlin, Germany}
\email{kutyniok@math.tu-berlin.de}

\author{Victoria Paternostro}
\address{Universidad de Buenos Aires and 
IMAS-CONICET, Consejo  Nacional de Investigaciones Cient\'ificas y T\'ecnicas, 1428 Buenos Aires,  Argentina}
\email{vpater@dm.uba.ar}

\author{Friedrich Philipp}
\address{Universidad de Buenos Aires, Ciudad Universitaria, Pabell\'on I, 1428 Buenos Aires,  Argentina}
\email{fmphilipp@gmail.com}

%%%%%%%%%%%%%%%%%%%%%%%%%%%%      ABSTRACT      %%%%%%%%%%%%%%%%%%%%%%%%%%%%%
\begin{abstract}
Fusion frames, and, more generally, operator-valued frame sequences are generalizations of classical frames,
which are today a standard notion when redundant, yet stable sequences are required. However, the question
of stability of duals with respect to perturbations has not been satisfactorily answered. In this paper, we quantitatively measure this stability by considering the associated deviations of the canonical and alternate dual
sequences from the original ones. It is proven that operator-valued frame sequences are indeed stable in this sense. 
Along the way, we also generalize existing definitions for fusion frame duals to the infinite-dimensional situation and analyze how they perform with respect to a list of desiderata which, to our minds, a fusion frame dual should satisfy. Finally, we prove a similar stability result as above for fusion frames and their canonical duals.\end{abstract}
%%%%%%%%%%%%%%%%%%%%%%%%%%%%%%%%%%%%%%%%%%%%%%%%%%%%%%%%%%%%%%%%%%%%%%%%%%%%%



\thanks{G.K. was supported by the Einstein Foundation Berlin, by the Einstein Center for Mathematics Berlin (ECMath), 
by Deutsche Forschungsgemeinschaft (DFG) SPP 1798, by the DFG Collaborative Research Center TRR 109 ``Discretization 
in Geometry and Dynamics'', and by the DFG Research Center {\sc Matheon} ``Mathematics for key technologies'' in Berlin. V.P. was supported by a fellowship for postdoctoral researchers from the Alexander von Humboldt Foundation and by Grants UBACyT  2002013010022BA and CONICET-PIP 11220110101018
 F.P. gratefully acknowledges support from MinCyT Argentina under grant PICT-2014-1480.}
      
\maketitle
\thispagestyle{empty}

\section{Introduction}\vspace{.5cm}
Introduced in 1952 by Duffin and Schaeffer \cite{DS52}, frames as an extension of the concept of orthonormal
bases allowing for redundancy, while still maintaining stability properties, are today a standard notion in 
mathematics and engineering. Applications range from more theoretical problems such as the Kadison-Singer Problem
\cite{CT06} and tensor decomposition \cite{ORS15} over questions inspired by (sparse) approximation theory \cite{KLL12} to real-world problems such as wireless communication and coding theory \cite{SH03}, quantum mechanics \cite{EF02}, and inverse scattering problems \cite{KMP15}. Recently, motivated by applications and also theoretical goals, generalizations of this framework have been developed: g-frames \cite{s}, operator-valued frames \cite{klz}, and fusion frames \cite{ckl}.

Reconstruction of the original vector from frame, fusion frame, or more general measurements, is typically
achieved by using a so-called (alternate or canonical) dual system. While having ensured stability of the 
measurement process already in the definition, the question of stability of the reconstruction with respect to perturbations of the frame or generalizations of this concept is more involved. The most natural approach to quantitatively measure this stability is by considering the associated deviations of the dual systems. However, for instance in the fusion frame setting, there are several classes and definitions of duality in the literature and questions related to stability of these duals with respect to perturbations are completely open so far. These are
the problems we tackle in this paper.

%************************************************************************************************************
\subsection{Frames, Fusion Frames, and Beyond}
%************************************************************************************************************

The two key properties of frames are redundancy and stability, which can easily be seen by their definition. 
A (not necessarily orthogonal) sequence $\Phi = (\varphi_i)_{i \in I}$ in a Hilbert space $\calH$ forms a {\it frame}, if 
it exhibits a norm equivalence $\|(\<\: \cdot \:,\varphi_i\>)_{i \in I}\|_{\ell^2(I)} \asymp \| \cdot \|_\calH$.
The associated analysis operator -- allowing the analysis of a vector -- is defined by $T_\Phi : \calH \to 
\ell^2(I)$, $x \mapsto (\<x,\varphi_i\>)_{i \in I}$. 

An analysis operator can be regarded as a collection of one-dimensional projections, leading in a natural way to so-called
fusion frames as a generalization of frames, serving, in particular, applications under distributed 
processing requirements \cite{ckl}. A {\it fusion frame} is a sequence $\calW = ((W_i,c_i))_{i \in I}$ of pairs of closed subspaces and weights, again exhibiting a norm equivalence $\|(c_i P_{W_i}( \cdot))_{i \in I}\|_{\bigoplus_{i \in I}W_i} \asymp \| \cdot \|_\calH$ with $P_{W_i}$ denoting the respective orthogonal projection and analysis operator being given by $T_\calW : \calH\to\bigoplus_{i \in I}W_i$, $x \mapsto (c_i P_{W_i}( \cdot))_{i \in I}$. Increasing the flexibility level once more and aiming for a thorough theoretical understanding, the weights and orthogonal projections are replaced by general operators $(A_i)_{i\in I}$ with $A_i\in \calB(\calH, \calK)$, $i\in I$, $\calK$ a Hilbert space, leading to {\it operator-valued frames} \cite{klz} and to the equivalent notion of {\it g-frames} \cite{s}. 

%************************************************************************************************************
\subsection{Reconstruction, Expansion, and Duals}
%************************************************************************************************************

At the heart of frame and fusion frame theory as well as their extensions is the problem of reconstructing 
the original vector after its analysis, i.e., after applying the analysis operator, which can also be regarded as a measurement operator or sampling operator. In frame theory,
the reconstruction formula takes the shape of
\[
x = \sum_{i \in I} \<x,\varphi_i\> \tilde{\varphi_i} = \sum_{i \in I} \<x,\tilde{\varphi_i}\> \varphi_i,\quad x \in \calH,
\]
where $(\tilde{\varphi_i})_{i \in I}$ is a so-called {\it {\rm (}alternate{\rm )} dual frame}. As can be seen, the 
second part of this formula even allows an expansion into the frame with a closed form sequence of
coefficients. Notice that this is not self-evident due to the redundancy of the frame. The {\it canonical 
dual frame} is a specific dual frame, which exhibits a closed form expression.

In the fusion frame setting, one certainly aims for a sequence with similarly advantageous properties, which
one might combine in the following list of desiderata for fusion frame duals:
\begin{enumerate}
\item[{\bf (D1)}] Reconstruction of any $x \in \calH$ from $T_\calW x$ possible.
\item[{\bf (D2)}] Proper generalization of alternate dual frames.
\item[{\bf (D3)}] Constitute a fusion frame themselves.
\item[{\bf (D4)}] Proper generalization of the canonical dual frame.
\end{enumerate}

The only approaches so far to introduce fusion frame duals can be found in \cite{g}, \cite{hmbz}, and \cite{hm}, where the latter two appeared several years after fusion frames were introduced in \cite{ck}. This already shows the delicacy of the fusion frame setting (noting that the definition of a dual frame is in fact rather simple). Here, it is our objective to consider, in particular, the effect of fusion frame perturbations on their duals. For this reason, we check how the particular definitions of fusion frame duals perform with respect to the above desiderata. As a result, we find that the duals from \cite{hm} and \cite{hmbz} satisfy the requirements, in contrast to the definition in \cite{g}. However, we also argue that a particular definition from \cite{hm} is the most appropriate to our minds, raise it to the infinite-dimensional situation and modify it slightly, culminating in the definition of fusion frame duals which we shall work with in this paper.

Interestingly, the definition of a dual in the setting of operator-valued frames or the slightly more general setting of operator-valued frame sequences is not that delicate. The reason for this phenomenon is that it is now only required that the dual shall constitute an operator-valued frame sequence instead of such a special sequence as a fusion frame (cf. {\bf (D3)}). This was also noticed in \cite{hm} as a motivation for a new definition of fusion frame duals.

%************************************************************************************************************
\subsection{Analysis of Stability}
%************************************************************************************************************

While a frame itself, its analysis operator being continuous and bounded below, provides stability with respect to both the measuring and the reconstruction process, it is not clear how a small perturbation of a frame, an operator-valued frame, or a fusion frame effects the associated set of duals. Here, we consider so-called {\it $\mu$-perturbations} which are (operator-valued) frames or frame sequences whose anaysis operator does not differ more than $\mu > 0$ from the analysis operator of the original object in norm. In the situation of frames, this is a well-studied subject (see, e.g., \cite{Bal97,CC97,CLZZ06,CH97,FS06,hmp}).

As discussed before, exact reconstruction or expansion is a crucial property of frames and their extensions, 
which in turn depends heavily on (alternate or canonical) dual sequences. Thus, in this paper, we study the
effect of $\mu$-perturbations on those dual sequences. To be precise, we consider a fixed triple consisting of an operator-valued frame sequence, a perturbation of it, and one of its duals and ask how close the duals of the perturbed sequence are to the original dual. Not only will this provide a very clear picture of this type of stability, but also allow us a deep understanding of the relation between the dual sequences and the original operator-valued frame sequences.

Results on the perturbation effect on the canonical duals of frame sequences can be found in \cite{hmp}. Perturbations of sequences beyond the frame setting have hardly been studied before. Except the works \cite{ag13} and \cite{s07} on perturbations of g-frames nothing can be found in this direction. Concerning fusion frames, the only work in this direction seems to be the paper \cite{ckl}, where the authors consider the stability of the fusion frame property in terms of a slightly different notion of perturbation.

%analyze erasures of subspaces \cite{CK08, hm}, which already required a quite subtle treatment.

%************************************************************************************************************
\subsection{Our Contribution}
%************************************************************************************************************

Our contributions are two-fold. First, we present a complete description of the general setting of operator-valued frame sequences and fusion frames and consider duals in both settings. In particular, we prove a parametrization of the duals of operator-valued frame sequences in terms of their analysis operators. In contrast to the operator-valued frame case (see \cite{klz} and Definition \ref{d:dual_ov}) the definition of fusion frame duals does not allow a straightforward generalization from the frame setting and thus requires a much more delicate handling. Therefore, we give an overview of the existing definitions of fusion frame duals from \cite{g,hm,hmbz} and discuss their performance with respect to the list of desiderata {\bf (D1)}--{\bf (D4)}. As a result of this discussion we define a version of dual fusion frames (Definition \ref{d:dual_fus}) which, in the finite-dimensional case, is very close to that of the ``block diagonal dual fusion frames'' from \cite{hm}. Additionally, we give a characterization of fusion frame duals in Theorem \ref{t:charac_dual_fus}.

Second, we provide a comprehensive perturbation analysis of both fusion frames and general operator-valued frame sequences in terms of the effect on their (alternate and canonical) duals, thereby generalizing and significantly improving existing results. In Theorems \ref{t:canonicals}, \ref{t:o-v-duals_sequences}, as well as \ref{thm:effect_pert_FF}, we show that indeed stability can be achieved in these situations and derive precise error estimates. The perturbation results Theorem \ref{t:canonicals} and Theorem \ref{t:o-v-duals_sequences} deal with the operator-valued frame sequence case whereas the fusion frame situation is tackled in Theorem \ref{thm:effect_pert_FF}. In the operator-valued frame case, we show in Theorem \ref{t:pert_frame_duals} that the dual, chosen in Theorem \ref{t:o-v-duals_sequences} to achieve stability, is a best approximation of the original dual among the set of duals of the perturbation. The big difference between the definitions of fusion frame duals and duals of operator-valued frames leads to the fact that the perturbation problem for fusion frames is much more delicate and requires another type of technique.

%************************************************************************************************************
\subsection{Outline}
%************************************************************************************************************

The paper is organized as follows. Section \ref{sec:OV_Sequences} is devoted to the introduction of operator-valued Bessel and frame sequences extending and slightly deviating from \cite{klz}. Vector frames and fusion frames are discussed as special cases. Canonical and alternate duals are then introduced in Section \ref{sec:Duals} and their key properties analyzed. First, in Subsection \ref{ss:duals}, duals are defined and studied in the general setting of operator-valued frame sequences. This is followed by the study of several notions of fusion frame duals and their performance with respect to our list of desired properties (see Subsection \ref{subsec:DualFF}). The last two sections focus on the impact of perturbations of the initial sequences on their duals, both in the general setting (Section \ref{sec:Perturbations}) and in the fusion frame setting (Section \ref{sec:Perturbations_FF}).

\subsection{Notation}
We close this introduction by fixing the notation we will use. 
The set of all bounded and everywhere defined linear operators between two Hilbert spaces $\calH$ and $\calK$ will be denoted by $\calB(\calH,\calK)$. As ususal, we set $\calB(\calH) := \calB(\calH,\calH)$. The norm on $\calB(\calH,\calK)$ will be the usual operator norm, i.e.
$$
\|T\| := \sup\left\{\|Tx\| : x\in\calH,\,\|x\|=1\right\}.
$$
We denote the range (i.e., the image) and the kernel (i.e., the null space) of $T\in\calB(\calH,\calK)$ by $\ran T$ and $\ker T$, respectively. The restriction of an operator $T\in\calB(\calH,\calK)$ to a subspace $V\subset\calH$ will be denoted by $T|V$. If $V$ is closed, by $P_V$ we denote the orthogonal projection onto $V$ in $\calH$ and by $I_V$ the identity operator on $V$.

Throughout this paper, $I\subset\N$ stands for a finite or countable index set and  $\calH$ and $\calK$ always denote Hilbert spaces. Recall that the space of $\calH$-valued $\ell^2$-sequences over $I$, defined by
$$
\ell^2(I,\calH) := \left\{(x_i)_{i\in I} : x_i\in\calH\,\forall i\in I,\;\sum_{i\in I}\|x_i\|^2 < \infty\right\},
$$
is a Hilbert space with scalar product
$$
\big\<(x_i)_{i\in I},(y_i)_{i\in I}\big\> = \sum_{i\in I}\<x_i,y_i\>,\quad (x_i)_{i\in I},(y_i)_{i\in I}\in\ell^2(I,\calH).
$$
We shall often denote
$$
\frakH := \ell^2(I,\calH)
\quad\text{and}\quad
\frakK := \ell^2(I,\calK).
$$

An operator $T\in\calB(\calH,\calK)$ is called {\em bounded below} if there exists $c > 0$ such that $\|Tx\|\ge c\|x\|$ for all $x\in\calH$. In the sequel, we will frequently make use of the following well known operator theoretical lemma.

\begin{lem}\label{l:op1}
Let $\calH$ and $\calK$ be Hilbert spaces. Then for $T\in\calB(\calH,\calK)$ the following statements are equivalent.
\begin{enumerate}
\item[{\rm (i)}]   $T$ is injective and $\ran T$ is closed.
\item[{\rm (ii)}]  $T^*$ is surjective.
\item[{\rm (iii)}] $T$ is bounded below.
\end{enumerate}
\end{lem}

\vspace{.2cm}
\section{Operator-valued Sequences}\label{sec:OV_Sequences}
This section is devoted to recalling the definition of the main object we shall work with, namely {\it operator-valued frame sequences}. We shall introduce the associated operators (analysis, synthesis, and frame operator) and discuss how vector frames and fusion frames fit into this framework.

\subsection{Operator-valued Frame Sequences}
As mentioned before, the concept of operator-valued frames  was first introduced  and intensively studied in \cite{klz}. Here, we generalize this notion to operator-valued Bessel sequences.

\begin{defn}
Let $\calH$ and $\calK$ be Hilbert spaces. A sequence of operators $\calA=(A_i)_{i\in I}$ with $A_i\in \calB(\calH, \calK)$, $i\in I$, is said to be an {\em operator-valued} {\rm (}or $\calB(\calH,\calK)$-{\em valued}\,{\rm )} {\em Bessel sequence} if there exists $\beta>0$ such that
\begin{equation}\label{eq:op-valued-bessel}
\sum_{i\in I}\|A_ix\|^2\leq \beta\|x\|^2\quad \textrm{for all }x\in\calH.
\end{equation}
The bound $\beta$ is said to be a {\em Bessel bound} of $\calA$.
\end{defn}

The following characterization of operator-valued Bessel sequences can easily be proved by using \cite[Chapter VII, p. 263]{rs}.

\begin{lem}\label{l:charac}
Let $\calA = (A_i)_{i\in I}$ be a sequence of operators in $\calB(\calH,\calK)$ and $\beta > 0$. Then the following statements are equivalent.
\begin{enumerate}
\item[{\rm (i)}]  $\calA$ is an operator-valued Bessel sequence with Bessel bound $\beta$.
\item[{\rm (ii)}] The series
$$
\sum_{i\in I}A_i^*A_i
$$
converges in the strong operator topology\footnote{A sequence $(T_i)_{i\in I}\subset\calB(\calH)$ is said to converge in the strong operator topology to $T\in\calB(\calH)$ if $(T_ix)_{i\in I}$ converges to $Tx$ for every $x\in\calH$.} to a non-negative self-adjoint operator with norm $\le\beta$.
\end{enumerate}
\end{lem}

For any sequence of operators $\calA = (A_i)_{i\in I}\subset\calB(\calH,\calK)$ we set
\begin{equation}\label{e:HPA}
\calH_\calA := \cls\left\{\ran A_i^*:\, i\in I\right\}\qquad \textrm{and} \qquad P_\calA:=P_{\calH_\calA}.
\end{equation}
Given an operator-valued Bessel sequence $\calA=(A_i)_{i\in I}$, we define its associated {\it analysis operator} $T_\calA:\calH\to\frakK$ by
\begin{equation}\label{analysis-op}
T_\calA x:=(A_ix)_{i\in I},\quad x\in\calH,
\end{equation}
where, we recall, 
$$
\frakK := \ell^2(I,\calK).
$$
The relation \eqref{eq:op-valued-bessel} ensures that $T_\calA$ is well-defined and that it is an element of $\calB(\calH,\frakK)$ with $\|T_\calA\|\leq\sqrt{\beta}$. The adjoint operator $T_\calA^*$ of $T_\calA$ is called the {\it synthesis operator} of $\calA$, and it is easily seen that
\begin{equation}\label{e:adj}
T_\calA^*(z_i)_{i\in I} = \sum_{i\in I}A^*_iz_i,\quad (z_i)_{i\in I}\in\frakK.
\end{equation}
For this, we only have to show that the series $\sum_{i\in I}A_i^*z_i$ converges in $\calH$. But this is (in the case $I = \N$) seen from
\begin{align*}
\left\|\sum_{i=m+1}^n A_i^*z_i\right\|^2
&= \sup_{\|x\|=1}\left|\left\<x,\sum_{i=m+1}^n A_i^*z_i\right\>\right|^2\le\left(\sup_{\|x\|=1}\sum_{i=m+1}^n|\<A_ix,z_i\>|\right)^2\\
&\le\sup_{\|x\|=1}\left(\sum_{i=m+1}^n\|A_ix\|^2\right)\left(\sum_{i=m+1}^n\|z_i\|^2\right)\le\beta\left(\sum_{i=m+1}^n\|z_i\|^2\right).
\end{align*}

\smallskip
\begin{lem}\label{l:ranges}
Let $\calA = (A_i)_{i\in I}$ be a $\calB(\calH,\calK)$-valued Bessel sequence. Then
\begin{equation}\label{e:TH}
\ol{\ran T_\calA^*} = \calH_\calA\quad\text{and}\quad\ker T_\calA = \calH_\calA^\perp.
\end{equation}
\end{lem}
\begin{proof}
Indeed, we have
$$
\calH_\calA^\perp = \left(\cls\{\ran A_i^*:\, i\in I\}\right)^\perp = \bigcap_{i\in I}\left(\ran A_i^*\right)^\perp = \bigcap_{i\in I}\ker A_i = \ker T_\calA.
$$
The first relation follows from this.
\end{proof}

The operator  $S_\calA:=T_\calA^*T_\calA = \sum_{i\in I}A_i^*A_i$ (the series converging in the strong operator topology) is called  the {\it frame operator} corresponding to $\calA$. It follows from Lemma \ref{l:charac} that $S_\calA$ is a bounded non-negative self-adjoint operator in $\calH$. Moreover, \eqref{e:TH} implies that $\calH_\calA$ is invariant under $S_\calA$ and that
$$
\<S_\calA x,x\> = \|T_\calA x\|^2 = \sum_{i\in I}\|A_i x\|^2 > 0\quad\text{for all }x\in\calH_\calA\setminus\{0\}.
$$

\begin{defn}
A sequence $\calA = (A_i)_{i\in I}$ of operators in $\calB(\calH,\calK)$ is called an {\em operator-valued} {\rm (}or $\calB(\calH,\calK)$-{\em valued}\,{\rm )} {\em frame sequence} if there exist $\alpha,\beta > 0$ such that
\begin{equation}\label{e:op-valued-frame}
\alpha\|x\|^2\,\le\,\sum_{i\in I}\|A_ix\|^2\,\le\,\beta\|x\|^2\quad\text{for all }x\in\calH_\calA.
\end{equation}
The constants $\alpha$ and $\beta$ are called {\em lower} and {\em upper frame bound} of $\calA$, respectively. If $\alpha = \beta$ is possible, $\calA$ is said to be {\em tight}. If even $\alpha = \beta = 1$, then $\calA$ is called an operator-valued {\em Parseval} frame sequence. If $\calH_\calA = \calH$ we say that $\calA$ is an {\em operator-valued frame} for $\calH$.
\end{defn}

An operator-valued frame sequence $\calA = (A_i)_{i\in I}$ with upper frame bound $\beta$ is an operator-valued Bessel sequence with Bessel bound $\beta$ since for $x\in\calH$ we have $A_i(I - P_\calA)x = 0$ for each $i\in I$ (see \eqref{e:TH}) and hence
$$
\sum_{i\in I}\|A_ix\|^2 = \sum_{i\in I}\|A_iP_\calA x\|^2\le\beta\|P_\calA x\|^2\le\beta\|x\|^2.
$$
Hence, $S_\calA$ is defined, and \eqref{e:op-valued-frame} is equivalent to $\alpha I_{\calH_\calA}\le S_\calA|\calH_\calA\le\beta I_{\calH_\calA}$. In particular, $S_\calA|\calH_\calA$ is boundedly invertible.

For each $j\in I$ we define the canonical embedding $\frakE_j : \calK\to\frakK$ by
$$
\frakE_jz := (\delta_{ij}z)_{i\in I},\quad z\in\calK.
$$
Its adjoint $\frakE_j^* : \frakK\to\calK$ is given by
$$
\frakE_j^*(z_i)_{i\in I} = z_j,\quad (z_i)_{i\in I}\in\frakK.
$$
Hence, for an operator-valued Bessel sequence $\calA=(A_i)_{i\in I}\subset\calB(\calH,\calK)$ we have
$$
\frakE_j^*T_\calA = A_j.
$$
From this, it follows that $T_\calA$ determines the operator-valued Bessel sequence $\calA$ uni\-que\-ly. Therefore, we will often identify an operator-valued Bessel sequence and its analysis operator. Even more holds: each $T\in\calB(\calH,\frakK)$ is the analysis operator of a $\calB(\calH,\calK)$-valued Bessel sequence. Indeed, for each $i\in I$, define $A_i\in\calB(\calH, \calK)$ by $A_i:=\frakE_i^*T$. Then, we have $A_i x = \frakE_i^*Tx = (Tx)_i$, $i\in I$, and thus 
$$
\sum_{i\in I}\|A_ix\|^2 = \|Tx\|^2\le\|T\|^2\|x\|^2,\quad x\in\calH,
$$
which implies that $\calA=(A_i)_{i\in I}$  is an operator-valued Bessel sequence. Moreover, for $x\in\calH$ we have $Tx = (A_ix)_{i\in I} = T_\calA x$, hence, $T = T_\calA$. In particular, the (linear) mapping $\calA\mapsto T_\calA$ between the linear space of operator-valued Bessel sequences $\frakB(I,\calH,\calK)$, indexed by $I$, and $\calB(\calH,\frakK)$ is bijective. With the norm $\|\calA\|:=\|T_\calA\|$ on $\frakB(I,\calH,\calK)$ it even becomes unitary.\label{page:Besselnorm}

The next lemma is the analogue of Corollary 5.5.3. in \cite{c}.

\begin{lem}\label{l:easy_op-valued-frames}
Let $T\in\calB(\calH,\frakK)$. Then the following statements hold:
\begin{itemize}
\item[{\rm (i)}]  $T$ is the analysis operator of an operator-valued frame sequence if and only if $\ran T^*$ is closed.
\item[{\rm (ii)}] $T$ is the analysis operator of an operator-valued frame if and only if $T^*$ is surjective.
\end{itemize}
\end{lem}
\begin{proof}
Due to the above discussion, we have $T = T_\calA$, where $\calA:=(\frakE_i^*T)_{i\in I}$. Moreover, by the closed range theorem (see, e.g., \cite[Theorem IV-5.13]{k}), $\ran T$ is closed if and only if $\ran T^*$ is closed.

(i). 
By definition, $\calA$ is an operator-valued frame sequence if and only if $\hat T := T_\calA|\calH_\calA$ is bounded below. By Lemma \ref{l:op1} this is the case if and only if $\hat T$ is injective and $\ran\hat T$ is closed. By \eqref{e:TH}, $\hat T$ is always injective and $\ran\hat T = \ran T_\calA = \ran T$. Hence, $\calA$ is an operator-valued frame sequence if and only if $\ran T$ is closed.

(ii). 
By (i) and \eqref{e:TH}, $\calA$ is an operator-valued frame if and only if $\ran T_\calA$ is closed and $\ker T_\calA = \{0\}$. Due to Lemma \ref{l:op1}, this holds if and only if $T^* = T_\calA^*$ is surjective.
\end{proof}

\begin{rem}
{\rm (a)} The definition of the analysis operator in \cite{klz} slightly differs from the one we give here, since it is defined to be an operator in $\calB(\calH, \ell^2(I)\otimes\calK)$ (see \cite[Proposition 2.3]{klz}). However, it can be seen that the spaces  $\ell^2(I)\otimes\calK$ and $\frakK$ are isometrically isomorphic through the mapping $\Upsilon: \ell^2(I)\otimes\calK\to\frakK$, $\Upsilon(e_j\otimes z)=\frakE_j(z)$ where $(e_i)_{i\in I}$ is the standard basis of $\ell^2(I)$. Moreover, for the (analysis) operator  $\theta_\calA$ defined in \cite[Eq. (6)]{klz}, one has $\Upsilon\theta_\calA=T_\calA$. Here, we prefer to work with the analysis operator as defined in \eqref{analysis-op} since it is  more suitable for our purposes and also it is a natural extension of the  analysis operator associated to a vector frame. 

\smallskip
{\rm (b)} In \cite{s}, Sun introduced the slightly more general concept of $G$-frames. $G$-frames are sequences of operators $A_i\in\calB(\calH,\calK_i)$ between Hilbert spaces $\calH$ and $\calK_i$, satisfying \eqref{e:op-valued-frame} for every $x\in\calH$. In principle, $\calK_i$ could be different from $\calK_j$ for $i\neq j$. However, in \cite{s}, Sun pointed out that for any sequence of Hilbert spaces $\calK_i$ one can always find a Hilbert space $\calK$ containing all $\calK_i$, namely $\calK = \bigoplus_{i\in I}\calK_i$. In this sense, every $G$-frame is also an operator-valued frame. 
\end{rem}

\subsection{Vector Frames as Special Operator-valued Frames}\label{ss:compare}
Recall that a sequence (of vectors) $\Phi=(\vphi_i)_{i\in I}$ in a separable Hilbert space $\calH$ is called a {\em Bessel sequence} if there exists $\beta > 0$ such that
$$
\sum_{i\in I}|\<x,\vphi_i\>|^2\,\le\,\beta\|x\|^2\quad\text{for all }x\in\calH.
$$
The {\em analysis operator} $T_\Phi : \calH\to\ell^2(I)$ corresponding to a Bessel sequence $\Phi = (\vphi_i)_{i\in I}$ is defined by
$$
T_\Phi x := (\<x,\vphi_i\>)_{i\in I},\quad x\in\calH.
$$
It is well known that $T_\Phi$ is bounded with norm $\|T_\Phi\|\le \sqrt \beta$. The adjoint $T_\Phi^* : \ell^2(I)\to\calH$ of $T_\Phi$ is called the {\em synthesis operator} corresponding to $\Phi$ and it is given by 
$$
T_\Phi^*c = \sum_{i\in I}c_i\vphi_i,\quad c\in\ell^2(I).
$$
The operator $S_\Phi := T_\Phi^*T_\Phi$ is called the {\em frame operator} corresponding to $\Phi$ and it is a non-negative bounded selfadjoint operator.
A sequence $\Phi = (\vphi_i)_{i\in I}$ in $\calH$ is called a {\em frame sequence} in $\calH$ if there exist $\alpha,\beta  > 0$ such that
\begin{equation*}\label{e:frame}
\alpha\|x\|^2\,\le\,\sum_{i\in I}|\<x,\vphi_i\>|^2\,\le\,\beta \|x\|^2\quad\text{for all }x\in\calH_\Phi, 
\end{equation*}
where
$$
\calH_\Phi=\cls\{\vphi_i:\,i\in I\}.
$$ 
A frame sequence $\Phi$ in $\calH$ is called a {\em frame for $\calH$} if $\calH_\Phi = \calH$. Consequently, a vector sequence is a frame sequence if and only if it is a frame for its closed linear span.

Given a Bessel sequence $\Phi=(\vphi_i)_{i\in I}$ in $\calH$, for every $i\in I$ we define an operator $A_i\in \calB(\calH, \C)$ by $A_ix:=\<x,\vphi_i\>$, $x\in\calH$. Thus, it is clear that $\calA = (A_i)_{i\in I}$ is an operator-valued Bessel sequence with Bessel bound $\beta$. Noticing that $A^*_ic = c\vphi_i$ for $c\in\C$ and that $\ell^2(I,\C) = \ell^2(I)$, we have that the analysis operator associated with $\calA$ coincides with the usual analysis operator corresponding to $\Phi$:
$$
T_\calA x=(A_ix)_{i\in I}=(\<x,\vphi_i\>)_{i\in I} = T_\Phi x,\quad x\in\calH.
$$
Consequently, we also have that $T_\calA^* = T_\Phi^*$ and $S_\calA = S_\Phi$. It is also clear that
$$
\calH_\calA = \cls\{\ran A^*_i:\,i\in I\} = \cls\{\vphi_i:\,i\in I\} = \calH_\Phi.
$$
Thus, the Bessel sequences in $\calH$ are exactly the $\calB(\calH,\C)$-valued Bessel sequences, so that operator-valued Bessel sequences naturally extend the notion of (vector) Bessel sequences. Evidently, an analogous correspondence holds for frame sequences (frames) and $\calB(\calH,\C)$-valued frame sequences (frames, respectively).

On the other hand, as it was noticed in \cite{klz}, an operator-valued frame sequence $\calA = (A_i)_{i\in I}\subset\calB(\calH,\calK)$ with $\dim\ran A_i=1$ for all $i\in I$ defines a (vector) frame sequence in $\calH$. Indeed, for each $i\in I$, let $e_i\in\calK$ be a unit vector such that $\ran A_i=\linspan\{e_i\}$. By the Riesz Representation Theorem, for every $i\in I$ there exists $\vphi_i\in\calH$ such that  $A_ix=\<x,\vphi_i\>e_i$ for all $x\in\calH$. Since  $A_i^*=\<\cdot, e_i\>\vphi_i$, we have that $\calH_\Phi=\calH_\calA$, where $\Phi=(\vphi_i)_{i\in I}$. Thus, \eqref{e:op-valued-frame} immediately yields that $\Phi$ is a frame sequence in $\calH$.

\subsection{Fusion Frames as Special Operator-valued Frames}\label{ss:fus_as_ov}
Let $(W_i)_{i\in I}$ be a sequence of closed subspaces of $\calH$ and $(c_i)_{i\in I}$ a sequence of non-negative real numbers such that for all $i\in I$ we have
\begin{equation}\label{e:spezial}
W_i = \{0\}\quad\Llra\quad c_i = 0.
\end{equation}
We shall call the sequence of pairs $((W_i,c_i))_{i\in I}$ a {\em fusion sequence}. A {\em fusion frame} for $\calH$ is a fusion sequence $((W_i,c_i))_{i\in I}$ for which there exist $\alpha,\beta > 0$ such that
$$
\alpha\|x\|^2\,\le\,\sum_{i\in I}c_i^2\|P_{W_i}x\|^2\,\le\,\beta\|x\|^2\quad\text{for all }x\in\calH.
$$
Fusion frames appeared for the first time in the literature in 2004 (cf.\ \cite{ck}) as ``frames of subspaces'' and were later renamed in \cite{ckl}. Obviously, a fusion sequence $\calW = ((W_i,c_i))_{i\in I}$ can be identified with the sequence of operators $\calA := (c_iP_{W_i})_{i\in I}$ which is a $\calB(\calH)$-valued frame for $\calH$ if and only if $\calW$ is a fusion frame for $\calH$. Therefore, the set of fusion frames for $\calH$ can be considered as a special (proper) subset of the $\calB(\calH)$-valued frames for $\calH$. We shall also call $\calW$ a {\em Bessel fusion sequence} ({\em fusion frame sequence}) whenever $\calA$ is a $\calB(\calH)$-valued Bessel sequence (frame sequence, resp.). The {\em analysis operator} and the {\em fusion frame operator} of the Bessel fusion sequence $\calW$ are then defined by $T_\calW := T_\calA$ and $S_\calW := S_\calA$, respectively. In accordance with \eqref{e:HPA}, we also define
$$
\calH_\calW := \cls\{W_i:\, i\in I\}\qquad\textrm{and}\qquad P_\calW:=P_{\calH_\calW}.
$$
At this point we would like to remark that in previous works (see, e.g., \cite{ck,ckl}) the analysis operator $T_\calW$ was not considered as an operator from $\calH$ to $\frakH = \ell^2(I,\calH)$ but from $\calH$ to $\bigoplus_{i\in I}W_i\subset\frakH$.

\vspace{.5cm}
\section{Duals of Operator-valued Frame Sequences and Fusion Frames}\label{sec:Duals}
In this section we describe and investigate the concept of duals of operator-valued frame sequences and provide a useful parametrization for the set of all duals of a given operator-valued frame sequence. As we shall see, operator-valued duals of fusion frames are in general not fusion frames. Therefore, we will discuss and compare the notions of duality for fusion frames existing in the literature and agree on the one that we  will consider in this paper.

\subsection{Duals of Operator-valued Frame Sequences}\label{ss:duals}
In the case of vector sequences there exist several notions of duals of frame sequences (cf.\ \cite{ce,hg} and also \cite{hkl}). In what follows, we define duals of operator-valued frame sequences, thereby generalizing the notion of the so-called Type I duals of vector frame sequences from \cite{hg}.

\begin{defn}\label{d:dual_ov}
Let $\calA = (A_i)_{i\in I}\subset\calB(\calH, \calK)$ be an operator-valued frame sequence and $\wt\calA = (\wt A_i)_{i\in I}\subset\calB(\calH,\calK)$ an operator-valued Bessel sequence. We say that $\wt\calA$ is a {\em dual operator-valued frame sequence} (or simply a {\em dual}\,) of $\calA$ if 
\begin{equation}\label{e:o-v-recform}
\calH_{\wt\calA}\subset\calH_\calA
\qquad\text{and}\qquad
\sum_{i\in I}\wt A_i^* A_ix = x\quad\text{for all }x\in\calH_\calA.
\end{equation}
\end{defn}

It is immediately seen that \eqref{e:o-v-recform} is equivalent to
$$
\ran T_{\wt\calA}^*\subset \calH_\calA
\qquad\text{and}\qquad
T_{\wt\calA}^*T_\calA = P_\calA.
$$
And as the above inclusion can equivalently be replaced by an equality, it follows from Lemma \ref{l:easy_op-valued-frames} that a dual is itself an operator-valued frame sequence.

By $\calD(\calA)$ we denote the set of all duals of the operator-valued frame sequence $\calA$ (which we identify with their analysis operators), that is,
$$
\calD(\calA) := \left\{T\in\calB(\calH,\frakK) : T^*T_\calA = P_\calA,\;\ran T^*\subset\calH_\calA\right\}.
$$
Among all duals of $\calA$ there is the so-called {\em canonical dual} which will play a special role in the sequel. It is defined by
\begin{equation}\label{e:canonical}
\left(A_i(S_\calA|\calH_\calA)^{-1}P_\calA\right)_{i\in I}.
\end{equation}
It is easily seen that this is indeed a dual of $\calA$ and that its analysis operator is given by
$$
T_\calA(S_\calA|\calH_\calA)^{-1}P_\calA.
$$
In analogy to the vector frame sequence case, we call the remaining duals of $\calA$ {\em alternate duals}.

The following lemma provides a characterization of the duals of $\calA$ in terms of their analysis operators. It can be viewed as an operator theoretical variant of a classical result in \cite{Li95} (see also \cite{c}).

\begin{lem}\label{l:o-v-dual_charac}
Let $\calA = (A_i)_{i\in I}\subset\calB(\calH, \calK)$ be an operator-valued frame sequence. Then
$$
\calD(\calA) = \left\{(T_\calA(S_\calA|\calH_\calA)^{-1} + L)P_\calA \,:\, L\in\calB(\calH_\calA,\frakK),\;L^*T_\calA = 0\right\}.
$$
\end{lem}
\begin{proof}
Let $T = (T_\calA(S_\calA|\calH_\calA)^{-1} + L)P_\calA$, where $L\in\calB(\calH_\calA,\frakK)$ is such that $L^*T_\calA = 0$. Then $T^* = (S_\calA|\calH_\calA)^{-1}T_\calA^* + L^*$, which implies $\ran T^*\subset\calH_\calA$, and
$$
T^*T_\calA = ((S_\calA|\calH_\calA)^{-1}T_\calA^* + L^*)T_\calA = (S_\calA|\calH_\calA)^{-1}T_\calA^*T_\calA = (S_\calA|\calH_\calA)^{-1}S_\calA = P_\calA,
$$
which proves $T\in\calD(\calA)$.

Conversely, let $T\in\calD(\calA)$, i.e.\ $\ran T^*\subset\calH_\calA$ and $T^*T_\calA = P_\calA$. Define the operator $L := (T|\calH_\calA) - T_\calA (S_\calA|\calH_\calA)^{-1}\in\calB(\calH_\calA,\frakK)$. Then we have
$$
L^*T_\calA = P_\calA T^*T_\calA - (S_\calA|\calH_\calA)^{-1}T_\calA^*T_\calA  = P_\calA  - P_\calA = 0,
$$
and $T|\calH_\calA = T_\calA(S_\calA|\calH_\calA)^{-1} + L$. Since $\ran T^*\subset\calH_\calA$ implies $\calH_\calA^\perp\subset\ker T$, we find that $T = (T_\calA(S_\calA|\calH_\calA)^{-1} + L)P_\calA$.
\end{proof}

In what follows we shall frequently make use of the following notation for an operator-valued frame sequence $\calA$:
$$
\calL_\calA := \left\{L\in\calB(\calH_\calA,\frakK) : T_\calA^*L = 0\right\}.
$$
It is obvious that $\calL_\calA$ is a closed linear subspace of $\calB(\calH_\calA,\frakK)$.

\begin{cor}
The set of duals $\calD(\calA)$ of a $\calB(\calH,\calK)$-valued frame sequence $\calA$ is a closed affine subspace of $\calB(\calH,\frakK)$.
\end{cor}
\begin{proof}
By Lemma \ref{l:o-v-dual_charac}, we have $\calD(\calA) = T_\calA(S_\calA|\calH_\calA)^{-1}P_\calA + \calL_\calA'$, where $\calL_\calA' := \{LP_\calA : L\in\calL_\calA\}\subset\calB(\calH,\frakK)$.
\end{proof}

\begin{rem}\label{r:HS}
If $\calH$ is finite-dimensional, $I$ is finite, and $\calA = (A_i)_{i\in I}\subset\calB(\calH,\calK)$ is an operator-valued frame for $\calH$, we have
$$
\calL_\calA = \left\{L\in\calB(\calH,\frakK) : T_\calA^*L = 0\right\}.
$$
Note that $\frakK = \calK^{|I|}$ in this case. The space $\calL_\calA$ is then perpendicular to $T_\calA S_\calA^{-1}$ with respect to the Hilbert-Schmidt scalar product
$$
\<X,Y\>_{\rm HS} = \tr(Y^*X), \quad X,Y\in\calB(\calH,\frakK),
$$
since for $L\in\calL_\calA$ we have $\<T_\calA S_\calA^{-1},L\>_{\rm HS} = \tr(L^*T_\calA S_\calA^{-1}) = 0$. Moreover, it is easily seen that the orthogonal projection onto $\calL_\calA$ (with respect to the inner product $\<\cdot\,,\cdot\>_{\rm HS}$) is then given by
\begin{equation}\label{e:op_fd}
P_{\calL_\calA}X = P_{\ker T_\calA^*}X,\qquad X\in\calB(\calH,\frakK).
\end{equation}
However, these observations cannot be generalized to the infinite-dimensional situation since $T_\calA S_\calA^{-1}$ cannot be Hilbert-Schmidt in this case. Indeed, if $T_\calA S_\calA^{-1}$ was Hilbert-Schmidt, then the operator $(T_\calA S_\calA^{-1})^*T_\calA S_\calA^{-1} = S_\calA^{-1}$ would be compact, implying that $S_\calA$ is not bounded. A contradiction.
\end{rem}

If\label{n:dual} $L\in\calL_\calA$ we denote by $\wt\calA(L)$ the dual of $\calA$ with the analysis operator
\begin{equation}\label{eq:A_tilde}
T_{\wt\calA(L)} = \left(T_\calA(S_\calA|\calH_\calA)^{-1} + L\right)P_\calA.
\end{equation}
Thus, $\wt\calA(0)$ is the canonical dual. Note that the mapping $\calL_\calA\to\calD(\calA)$, $L\mapsto\wt\calA(L)$, is one-to-one and therefore parametrizes the duals of $\calA$.

\begin{rem}
We mention that each operator in $\calL_\calA$ can be written in the form $P_{\ker T_\calA^*}M$ with $M\in\calB(\calH_\calA,\frakK)$. Since $I - P_{\ker T_\calA^*} = P_{\ran T_\calA} = T_\calA(S_\calA|\calH_\calA)^{-1}T_\calA^*$, we have
\begin{align*}
\calD(\calA)
&= \left\{\left[T_\calA(S_\calA|\calH_\calA)^{-1} + (I - T_\calA(S_\calA|\calH_\calA)^{-1}T_\calA^*)M\right]P_\calA \,:\, M\in\calB(\calH_\calA,\frakK)\right\}\\
&= \left\{\left[ T_\calA(S_\calA|\calH_\calA)^{-1}\left(I - T_\calA^*M\right) + M \right]P_\calA \,:\, M\in\calB(\calH_\calA,\frakK)\right\}.
\end{align*}
Although this representation of the duals of $\calA$ is slightly more explicit, it does not provide a parametrization as $M\mapsto P_{\ker T_\calA^*}M$ is not one-to-one.
\end{rem}

\begin{cor}
If $\calA$ is an operator-valued frame sequence and $L\in\calL_\calA$ then the frame operator of $\wt\calA(L)$ is given by
$$
S_{\wt\calA(L)} = \left((S_\calA|\calH_\calA)^{-1} + L^*L\right)P_\calA.
$$
\end{cor}
\begin{proof}
As $L^*T_\calA = 0$, and thus also $T_\calA^*L = 0$, we have
\begin{align*}
S_{\wt\calA(L)}
&= \big((T_\calA(S_\calA|\calH_\calA)^{-1} + L)P_\calA\big)^*(T_\calA(S_\calA|\calH_\calA)^{-1} + L)P_\calA\\
&= P_\calA\big((S_\calA|\calH_\calA)^{-1}T_\calA^* + L^*\big)(T_\calA(S_\calA|\calH_\calA)^{-1} + L)P_\calA\\
&= P_\calA\left[(S_\calA|\calH_\calA)^{-1}T_\calA^*T_\calA(S_\calA|\calH_\calA)^{-1} + L^*L\right]P_\calA\\
&= P_\calA\left((S_\calA|\calH_\calA)^{-1} + L^*L\right)P_\calA\\
&= ((S_\calA|\calH_\calA)^{-1} + L^*L)P_\calA,
\end{align*}
which proves the claim.
\end{proof}

\begin{rem}\label{r:bound-for-dual}
The frame operator of the canonical dual $\wt \calA= \wt\calA(0)$ of $\calA$ is given by $(S_\calA|\calH_\calA)^{-1}P_\calA$.  
This implies that if $0<\alpha\leq\beta$ are frame bounds for $\calA$, then $0<\beta^{-1}\leq\alpha^{-1}$ are frame bounds for $\wt\calA$. In particular, this last fact gives
\begin{equation}\label{e:la_tormenta}
\left\|T_\calA(S_\calA|\calH_\calA)^{-1}P_\calA\right\|\,\le\,\frac{1}{\sqrt\alpha},
\end{equation}
which we will use frequently below.
\end{rem}

In Subsection \ref{ss:fus_as_ov}, it was shown that fusion frames can be regarded as special operator-valued frames. However, as the next proposition shows, the operator-valued duals of fusion frames might themselves not correspond to fusion frames. We shall say that two operator-valued Bessel sequences $\calA$ and $\calB$ are orthogonal if $\calH_\calA\perp\calH_\calB$.

\begin{prop}\label{p:dualff}
For a fusion frame sequence $\calW = ((W_i,c_i))_{i\in I}$ in $\calH$ the following statements are equivalent:
\begin{enumerate}
\item[{\rm (i)}]  The canonical dual of $(c_iP_{W_i})_{i\in I}$ is a fusion frame sequence.
\item[{\rm (ii)}] $\calW$ is a union of mutually orthogonal tight fusion frame sequences.
\end{enumerate}
\end{prop}
\begin{proof}
The canonical dual of $(c_iP_{W_i})_{i\in I}$ is given by $(c_iP_{W_i}(S_\calW|\calH_\calW)^{-1}P_\calW)_{i\in I}$ (cf.\ \eqref{e:canonical}). It obviously corresponds to a fusion frame sequence if and only if for each $i\in I$ the operator $P_{W_i}(S_\calW|\calH_\calW)^{-1}$ coincides with a positive multiple of an orthogonal projection in $\calH_\calW$. As the range of this operator coincides with $W_i$, (i) is satisfied if and only if for each $i\in I$ there exists $d_i > 0$ such that 
$P_{W_i}(S_\calW|\calH_\calW)^{-1} = d_iP_{W_i}|\calH_\calW$, and, by adjunction,
$$
(S_\calW|\calH_\calW)^{-1}(P_{W_i}|\calH_\calW) = d_i (P_{W_i}|\calH_\calW)\quad\forall i\in I.
$$
This is equivalent to
\begin{equation}\label{e:a}
S_\calW P_{W_i} = d_i^{-1}P_{W_i}\quad\forall i\in I.
\end{equation}
Hence, (i) holds if and only if there exists $(d_i)_{i\in I}\subset(0,\infty)$ such that \eqref{e:a} holds.

(i)$\Sra$(ii). Let $(d_i)_{i\in I}\subset(0,\infty)$ be as in \eqref{e:a}. Define an equivalence relation $\sim$ on $I$ by $i_1\sim i_2\,:\Llra\;d_{i_1} = d_{i_2}$, $i_1,i_2\in I$. Let $J\subset I$ be a set of representatives of all the cosets in $I/\sim$, and for $j\in J$ put $I_j := [j]_\sim$. Then, $(I_j)_{j\in J}$ is a partition of $I$. For $j\in J$ we further define $\la_j := d_i^{-1}$ if $i\in I_j$ and $V_j := \cls\{W_i : i\in I_j\}$. Then \eqref{e:a} implies that $V_j\subset \ker (S_\calW -\la_j Id)$, and since eigenspaces of  self-adjoint operators are mutually orthogonal, we have that $V_j\perp V_k$ for $j\neq k$, $j,k\in J$. Hence, $((W_i,c_i))_{i\in I_j}$ and $((W_i,c_i))_{i\in I_k}$ are orthogonal for $j\neq k$, $j,k\in J$. It remains to show that for each $j\in J$ the sequence $\calW_j:=((W_i,c_i))_{i\in I_j}$ is a tight fusion frame sequence. For this, let $j\in J$ and $x\in V_j$. Then the tightness is seen by
\begin{align*}
S_{\calW_j}x = \sum_{i\in I_j}c_i^2P_{W_i}x = \sum_{k\in J}\sum_{i\in I_k}c_i^2P_{W_i}x = \sum_{i\in I}c_i^2P_{W_i}x = S_\calW x = \la_j x,
\end{align*}
where in the second equality we use that $(V_j)_{j\in J}$ are mutually orthogonal. 

(ii)$\Sra$(i). Due to (ii), there exist a partition $I = \bigcup_{j\in J} I_j$ of $I$ and $(\alpha_j)_{j\in J}\subset (0,\infty)$ such that $\calW_j := ((W_i,c_i))_{i\in I_j}$ is an $\alpha_j$-tight fusion frame sequence for each $j\in J$ and $\calW_j$ and $\calW_k$ are orthogonal for $j\neq k$. Put $V_j := \calH_{\calW_j}$. Then, by the tightness of the $\calW_j$, for $x\in\calH$ we have
$$
S_\calW x = \sum_{j\in J}\sum_{i\in I_j}c_i^2P_{W_i}x = \sum_{j\in J}S_{\calW_j}x = \sum_{j\in J}\alpha_jP_{V_j}x.
$$
For $i\in I$, let $j(i)\in J$ be such that $i\in I_{j(i)}$. Then the mutual orthogonality of the $\calW_j$ gives
$$
S_\calW P_{W_i} = \sum_{j\in J}\alpha_jP_{V_j}P_{W_i} = 
\alpha_{j(i)}P_{W_i},
$$
which is \eqref{e:a} for $d_i=\alpha_{j(i)}^{-1}$.
\end{proof}

\begin{rem}
Note that {\rm (ii)} in Proposition \rmref{p:dualff} implies that each $W_i$ is a subspace of some eigenspace of $S_\calW$ (cf.\ \eqref{e:a}).
\end{rem}

\subsection{Duals of Fusion Frames}\label{subsec:DualFF}
Since their introduction in 2004 (see \cite{ck}), fusion frames have been extensively studied. However, there have only been two approaches yet to define duals of fusion frames (see \cite{g,hmbz}, and also \cite{hm} for the finite-dimensional case). In fact, this is a delicate task since it quickly turns out that one has to give up on certain analogues to the vector frames case. 

The first proposal for what a dual fusion frame should be  was made by P.\ G\u avru\c ta in \cite{g}. He calls a Bessel Fusion sequence $\calV = ((V_i,d_i))_{i\in I}$ a dual of a fusion frame $\calW = ((W_i,c_i))_{i\in I}$ if
$$
\sum_{i\in I}c_id_iP_{V_i}S_\calW^{-1}P_{W_i}x = x\quad\text{for all }x\in\calH.
$$
Here, we shall call these fusion sequences {\em G\u avru\c ta duals}. Although it is pro\-ven in \cite{g} that a G\u avru\c ta dual $\calV$ of $\calW$ is itself a fusion frame, it is in general not true that, conversely, $\calW$ is a G\u avru\c ta dual of $\calV$. A simple counterexample is the following:
$$
\calH = \C^2,\quad\calW = ((\linspan\{e_i\},1))_{i=1}^2,\quad\calV = ((\C^2,1))_{i=1}^2.
$$
In fact, we have $S_\calW = I$ and $S_\calV = 2I$, and thus
$$
\sum_{i=1}^2c_id_iP_{V_i}S_\calW^{-1}P_{W_i} = I,
\quad\text{whereas}\quad
\sum_{i=1}^2c_id_iP_{W_i}S_\calV^{-1}P_{V_i} = \frac 1 2 I.
$$

Recently, in \cite{hmbz} (see also \cite{hm}), a Bessel fusion sequence $\calV = ((V_i,d_i))_{i\in I}$ was called a dual fusion frame of $\calW = ((W_i,c_i))_{i\in I}$ if there exists a bounded operator
$$
Q : \bigoplus_{i\in I}W_i\to\bigoplus_{i\in I}V_i
$$
such that
\begin{equation}\label{e:shit}
x = \sum_{i\in I}d_i\big(Q (c_jP_{W_j}x)_{j\in I}\big)_i,\quad x\in\calH.
\end{equation}
However, given a fusion frame for $\calH$, {\em every} fusion frame for $\calH$ is a corresponding dual fusion frame in this sense. Indeed, let $\calW = ((W_i,c_i))_{i\in I}$ and $\calV = ((V_i,d_i))_{i\in I}$ be fusion frames for $\calH$. Define
$$
Q := (T_\calV S_\calV^{-1}S_\calW^{-1}T_\calW^*)\Big|{\bigoplus_{i\in I}W_i}.
$$
Note that $\ran Q\subset\bigoplus_{i\in I}V_i$ so that $Q$ can be seen as an operator in $\calB(\bigoplus_{i}W_i,\bigoplus_{i}V_i)$. Now, for $x\in\calH$ we have (note that $(c_jP_{W_j}x)_{j\in I} = T_\calW x$)
\begin{align*}
\sum_{i\in I}d_i\big(Q (c_jP_{W_j}x)_{j\in I}\big)_i
&= \sum_{i\in I}d_i\big(T_\calV S_\calV^{-1}S_\calW^{-1}T_\calW^*T_\calW x\big)_i = \sum_{i\in I} d_i\left(T_\calV S_\calV^{-1}x\right)_i\\
&= \sum_{i\in I}d_i^2 P_{V_i} S_\calV^{-1} x = S_\calV^{-1}\sum_{i\in I}d_i^2 P_{V_i}x = S_\calV^{-1}S_\calV x = x.
\end{align*}
Thus, $\calV$ is a dual of $\calW$ in the sense of \cite{hmbz}.

This shows that there is too much freedom in the choice of the operator $Q$ in this definition and it seems reasonable to work with particular classes of operators $Q$. In fact, in \cite{hm,hmbz} the classes of ``component-preserving'' and ``block-diagonal'' duals were considered (the latter only in the finite-dimensional case) which arise from allowing only diagonal operators in the above definition. In this paper, we will consider normalized diagonal operators $Q$. In what follows, we shall explain and justify our choice in detail.

First, we find that the reconstruction formula \eqref{e:shit} -- in this general form -- seems to be of hardly any use in applications. However, if one restricts the set of ``admissible'' operators $Q$ to diagonal operators $Q = \diag(Q_i)_{i\in I}$ with $Q_i\in\calB(W_i,V_i)$ for each $i\in I$, formula \eqref{e:shit} becomes considerably simpler (cf.\ \cite{hm}):
\begin{equation*}
x = \sum_{i\in I}c_id_iQ_iP_{W_i}x,\quad x\in\calH,
\end{equation*}
Since here we prefer to work with $\frakH = \ell^2(I,\calH)$ instead of $\bigoplus_{i\in I} W_i$ and $\bigoplus_{i\in I} V_i$, we allow $Q_i\in\calB(\calH)$ for each $i\in I$ and ask for the validity of
\begin{equation}\label{e:recform_fus}
x = \sum_{i\in I}c_id_iP_{V_i}Q_iP_{W_i}x,\quad x\in\calH.
\end{equation}
Since the weights $(d_i)_{i\in I}$ are somewhat arbitrary in this version (if $(Q_i)_{i\in I}$ and $(d_i)_{i\in I}$ satisfy \eqref{e:recform_fus}, also $(\alpha_iQ_i)_{i\in I}$ and $(\alpha_i^{-1}d_i)_{i\in I}$ do for every bounded positive sequence $(\alpha_i)_{i\in I}$), we shall furthermore require that $\|Q_i\| = 1$ for each $i\in I$. Moreover, since in \eqref{e:recform_fus} only the action of $Q_i$ on $W_i$ is important and, in addition, only that part being further mapped to $V_i$, we shall also require that $W_i^\perp\subset\ker Q_i$ and $\ran Q_i\subset V_i$. Then \eqref{e:recform_fus} reduces to
\begin{equation}\label{e:recform_fus_reduce}
x = \sum_{i\in I}c_id_iQ_ix,\quad x\in\calH,
\end{equation}
since in this case we have $Q_i = P_{V_i}Q_i = Q_iP_{W_i}$.

For two fusion sequences $\calV = ((V_i,d_i))_{i\in I}$ and $\calW = ((W_i,c_i))_{i\in I}$ in $\calH$ we define (see \eqref{e:spezial})
$$
I_0(\calV,\calW) := \{i\in I : V_i = \{0\} \text{ or }W_i = \{0\}\} = \{i\in I : c_i = 0 \text{ or }d_i = 0\}.
$$

\begin{defn}\label{d:dual_fus}
Let $\calW = ((W_i,c_i))_{i\in I}$ be a fusion frame for $\calH$. A Bessel fusion sequence $\calV = ((V_i,d_i))_{i\in I}$ will be called a {\em dual fusion frame} {\rm (}or a {\em fusion frame dual} or, shortly, a {\em FF-dual}\,{\rm )} of $\calW$ if there exists a sequence $(Q_i)_{i\in I}\subset\calB(\calH)$ satisfying
\begin{equation}\label{e:require}
W_i^\perp\subset\ker Q_i, \quad\ran Q_i\subset V_i,\quad\text{and }\;\|Q_i\|=1\text{ if }i\notin I_0(\calV,\calW)
\end{equation}
for each $i\in I$ such that \eqref{e:recform_fus_reduce} holds.
\end{defn}

\begin{rem}\label{r:ffdual}
{\rm (a)} The first condition in \eqref{e:require} yields $\ran Q_i = Q_iW_i$. Hence, the second means $Q_iW_i\subset V_i$. If $V_i = \{0\}$ or $W_i = \{0\}$ the first two conditions in \eqref{e:require} imply $Q_i = 0$.

\smallskip
{\rm (b)} A dual fusion frame $((V_i,d_i))_{i\in I}$ as defined in \cite{hmbz} was called {\em component-preserving} if there exists a bounded sequence $(Q_i)_{i\in I}\subset\calB(\calH)$ satisfying \eqref{e:recform_fus} and $Q_iW_i = V_i$ \braces{and not only $Q_iW_i\subset V_i$} for each $i\in I$. Here, we shall not further study this subclass. If $((V_i,d_i))_{i\in I}$ is a fusion sequence in the finite-dimensional Hilbert space $\calH$ satisfying Definition \rmref{d:dual_fus} \braces{without the normalization condition}, then it is a {\em block-diagonal dual fusion frame} as defined in \cite{hm}.

\smallskip
{\rm (c)} If $\calV = ((V_i,d_i))_{i\in I}$ is a G\u avru\c ta dual of $\calW = ((W_i,c_i))_{i\in I}$ with the property that $P_{V_i}S_\calW^{-1}P_{W_i} = 0\,\Lra\,i\in I_0(\calV,\calW)$ for every $i\in I$, then the fusion sequence $((V_i,\|P_{V_i}S_\calW^{-1}P_{W_i}\|d_i))_{i\in I}$ is a fusion frame dual of $\calW$. Indeed, if for $i\in I$ we put $A_i := P_{V_i}S_\calW^{-1}P_{W_i}$ as well as $Q_i := A_i/\|A_i\|$ if $A_i\neq 0$ and $Q_i := 0$ otherwise, then $(Q_i)_{i\in I}$ satisfies \eqref{e:require}, and for $x\in\calH$ we have $x = \sum_{i\in I}c_id_iA_ix = \sum_{i\in I}c_i(\|A_i\|d_i)Q_ix$. See also \cite[Example 6.1]{hm}.
\end{rem}

After having presented the notions of duality from \cite{g,hmbz,hm} and the one in Definition \ref{d:dual_fus}, it seems convenient to us to enumerate a few desiderata which duals of a fusion frame $\calW = ((W_i,c_i))_{i\in I}$ should satisfy and then see how the different proposals stated above perform with respect to them:
\begin{enumerate}
\item[{\bf (D1)}] They should allow for reconstructing signals $x\in\calH$ from their fusion frame measurements $c_iP_{W_i}x$, $i\in I$.
\item[{\bf (D2)}] They should properly generalize the notion of dual (vector) frames, that is:
   \begin{enumerate}
   \item[(D2a)] If $(\psi_i)_{i\in I}$ is a dual of the frame $(\vphi_i)_{i\in I}$ for $\calH$, then $((\linspan\{\psi_i\},\|\psi_i\|))_{i\in I}$ is a dual fusion frame of $((\linspan\{\vphi_i\},\|\vphi_i\|))_{i\in I}$.
   \item[(D2b)] If $\dim W_i\in\{0,1\}$ for each $i\in I$ and $\calV = ((V_i,d_i))_{i\in I}$ is a dual fusion frame of $\calW$ with $\dim V_i\in\{0,1\}$ for each $i\in I$, then there exist vectors $\vphi_i\in W_i$ and $\psi_i\in V_i$ with $\|\vphi_i\| = c_i$ and $\|\psi_i\| = d_i$, $i\in I$, such that $(\psi_i)_{i\in I}$ is a dual frame of $(\vphi_i)_{i\in I}$.
   \end{enumerate}
\item[{\bf (D3)}] If $\calV$ is a dual fusion frame of $\calW$ then it should itself be a fusion frame and also $\calW$ should be a dual fusion frame of $\calV$.
\item[{\bf (D4)}] The fusion sequence $((S_\calW^{-1}W_i,c_i\|S_\calW^{-1}|W_i\|))_{i\in I}$ should be a dual fusion frame of $\calW$ (the canonical dual).
\end{enumerate}
Whereas {\bf (D1)}--{\bf (D3)} are evident requirements, the choice of the weights of the canonical dual fusion frame in {\bf (D4)} might not be clear a priori. To explain our choice, consider the canonical dual $(S_\Phi^{-1}\vphi_i)_{i\in I}$ of a (vector) frame $\Phi = (\vphi_i)_{i\in I}$. Translated to the fusion frame setting, we have $W_i = \linspan\{\vphi_i\}$ and $c_i = \|\vphi_i\|$ as well as $V_i = \linspan\{S_\Phi^{-1}\vphi_i\} = S_\Phi^{-1}W_i$ and $d_i = \|S_\Phi^{-1}\vphi_i\|$. Thus, if $c_i\neq 0$, the weights of the canonical dual are $d_i = c_i\|S_\Phi^{-1}(\vphi_i/\|\vphi_i\|)\| = c_i\|S_\Phi^{-1}|W_i\|$. The same trivially holds for $c_i = 0$.

As we have shown in the beginning of this subsection, G\u avru\c ta duals do not satisfy desideratum {\bf (D3)}. For completeness, we mention that if one might want to directly generalize the duality definition for (operator-valued) frames by requiring that $T_\calV^* T_\calW = I$ for a dual $\calV$ of $\calW$,  this definition would violate {\bf (D4)}. A simple example for this is given by
$$
\calH = \C^2,\; W_1 = \linspan\{e_1\},\;W_2 = \linspan\{e_1+e_2\},\;c_1 = c_2 = 1,
$$
where $\{e_1, e_2\}$ is the canonical standard basis of $\C^2$.

As it turns out, all definitions of fusion frame duals from \cite{hmbz,hm} as well as Definition \ref{d:dual_fus} satisfy the desiderata {\bf (D1)}--{\bf (D4)}. Indeed, this follows essentially from \cite[Sections 3 and 4]{hm} and \cite[Section 3]{hmbz}. However, for the convenience of the reader, we prove this claim for the FF-duals from Definition \ref{d:dual_fus}.

\begin{prop}\label{p:desiderata}
The desiderata {\bf (D1)}--{\bf (D4)} are satisfied for the notion of fusion frame duals defined as in Definition {\rm\ref{d:dual_fus}}.
\end{prop}
\begin{proof}
It is clear that the definition satisfies {\bf (D1)} (see \eqref{e:recform_fus}). Moreover, \eqref{e:recform_fus} is equivalent to $T_\calV^*QT_\calW = I$, where $Q = \bigoplus_{i\in I}Q_i$. This identity yields both $\ran T_\calV^* = \calH$ and $T_\calW^*Q^*T_\calV = I$. Therefore, $\calV$ is a fusion frame for $\calH$ (cf.\ Lemma \ref{l:easy_op-valued-frames}), and $\calW$ is a FF-dual of $\calV$, meaning that {\bf (D3)} is satisfied. To prove that {\bf (D4)} holds, we note that $((S_\calW^{-1}W_i,c_i))_{i\in I}$ is a G\u avru\c ta dual of $\calW$ by \cite{g}. It has the property in Remark \ref{r:ffdual} (c). Hence, $((S_\calW^{-1}W_i,c_i\|S_\calW^{-1}|W_i\|))_{i\in I}$ is a FF-dual of $\calW$ since $\|P_{S_\calW^{-1}W_i}S_\calW^{-1}P_{W_i}\| = \|S_\calW^{-1}P_{W_i}\| = \|S_\calW^{-1}|W_i\|$.

Let us see that also {\bf (D2)} is satisfied. For this, let $(\psi_i)_{i\in I}$ be a dual of the frame $(\vphi_i)_{i\in I}$ for $\calH$ as in (D2a). For $i\in I$, we put $W_i := \linspan\{\vphi_i\}$, $V_i := \linspan\{\psi_i\}$, and
$$
Q_ix :=
\begin{cases}
0 & \text{ if $\vphi_i = 0$ or $\psi_i = 0$}\\
\left\<x,\frac{\vphi_i}{\|\vphi_i\|}\right\>\frac{\psi_i}{\|\psi_i\|} & \text{ otherwise}
\end{cases},\qquad x\in\calH.
$$
Then $(Q_i)_{i\in I}$ satisfies \eqref{e:require}. Moreover, we have 
\begin{align*}
\sum_{i\in I}\|\vphi_i\|\|\psi_i\|Q_ix = \sum_{i\in I,\,\psi_i\neq 0,\vphi_i\neq 0}\|\vphi_i\|\|\psi_i\|\left\<x,\frac{\vphi_i}{\|\vphi_i\|}\right\>\frac{\psi_i}{\|\psi_i\|} = \sum_{i\in I}\<x,\vphi_i\>\psi_i = x.
\end{align*}
Hence, $((V_i,\|\psi_i\|))_{i\in I}$ is a FF-dual of $((W_i,\|\vphi_i\|))_{i\in I}$, as desired in (D2a).

For (D2b), let $(Q_i)_{i\in I}\subset\calB(\calH)$ be a sequence as in Definition \ref{d:dual_fus}. Choose $\vphi_i\in W_i$ with $\|\vphi_i\| = c_i$, $i\in I$. We put $\psi_i := c_i^{-1}d_iQ_i\vphi_i$ if $c_i\neq 0$. If $c_i = 0$ we choose an arbitrary $\psi_i\in V_i$ with $\|\psi_i\| = d_i$. Let us see that $\|\psi_i\| = d_i$ for each $i\in I$. This is clear if $c_i = 0$. Let $c_i\neq 0$. If $Q_i\neq 0$ then $\|Q_i\| = 1$ and hence $\|Q_i\vphi_i\| = c_i$, i.e., $\|\psi_i\| = d_i$. If $Q_i = 0$ then $i\in I_0(\calV,\calW)$, implying that $V_i = \{0\}$ as $c_i\neq 0$ yields $W_i\neq\{0\}$. But then $d_i = 0$ and thus $\|\psi_i\| = 0 = d_i$. For arbitrary $x\in\calH$ we have
$$
\sum_{i\in I}|\<x,\psi_i\>|^2 = \sum_{d_i\neq 0}|\<x,\psi_i\>|^2 = \sum_{d_i\neq 0}d_i^2\|P_{V_i}x\|^2 = \sum_{i\in I}d_i^2\|P_{V_i}x\|^2 = \|T_\calV x\|^2.
$$
This implies that $(\psi_i)_{i\in I}$ is a Bessel sequence in $\calH$. Finally,
\begin{align*}
\sum_{i\in I}\<x,\vphi_i\>\psi_i
&= \sum_{c_i\neq 0}c_i^{-1}d_i\<x,\vphi_i\>Q_i\vphi_i = \sum_{c_i\neq 0}c_i^{-1}d_iQ_i\big(\<x,\vphi_i\>\vphi_i\big)\\
&= \sum_{c_i\neq 0}c_id_iQ_iP_{W_i}x = \sum_{i\in I}c_id_iQ_ix = x.
\end{align*}
Hence, $(\psi_i)_{i\in I}$ is a dual frame of $(\vphi_i)_{i\in I}$.
\end{proof}

The following theorem provides a characterization of fusion frame duals. Note that implication (iii)$\Sra$(i) is essentially \cite[Lemma 3.5]{hmbz}.

\begin{thm}\label{t:charac_dual_fus}
Let $\calW = ((W_i,c_i))_{i\in I}$ be a fusion frame for $\calH$ and let $\calV = ((V_i,d_i))_{i\in I}$ be a Bessel fusion sequence. Then the following statements are equivalent:
\begin{enumerate}
\item[{\rm (i)}]   $\calV$ is a fusion frame dual of $\calW$.
\item[{\rm (ii)}]  There exists a sequence $(Q_i)_{i\in I}\subset\calB(\calH)$ satisfying \eqref{e:require} such that $(d_iQ_i^*)_{i\in I}$ is a $\calB(\calH)$-valued dual of $(c_iP_{W_i})_{i\in I}$.
\item[{\rm (iii)}] There exists a Bessel sequence $\calL = (L_i)_{i\in I}\subset\calB(\calH)$ with $T_\calL^*T_\calW = 0$ such that for the operators $A_i := (c_iS_\calW^{-1} + L_i^*)P_{W_i}$, $i\in I$, we have $\ran A_i\subset V_i$ and, if $i\notin I_0(\calV,\calW)$, $\|A_i\| = d_i$.
\end{enumerate}
\end{thm}
\begin{proof}
(i)$\Sra$(ii). Let $(Q_i)_{i\in I}\subset\calB(\calH)$ be as in Definition \ref{d:dual_fus}. Then this sequence satisfies \eqref{e:require} and $x = \sum_{i\in I}c_id_iQ_ix$ for all $x\in\calH$. Since $Q_i = P_{V_i}Q_i$ and $\|Q_i\|\le 1$ for each $i\in I$, we have for $x\in\calH$:
$$
\sum_{i\in I}\|d_iQ_i^*x\|^2 = \sum_{i\in I}d_i^2\|Q_i^*P_{V_i}x\|^2 \,\le\, \sum_{i\in I}d_i^2\|P_{V_i}x\|^2 = \|T_\calV x\|^2,
$$
which shows that $\calB := (d_iQ_i^*)_{i\in I}$ is a Bessel sequence. Moreover,
$$
T_\calB^*T_\calW x = \sum_{i\in I}d_iQ_i(c_iP_{W_i}x) = \sum_{i\in I}c_id_iQ_ix = x
$$
for $x\in\calH$. This proves (ii).

(ii)$\Sra$(iii). As before, put $\calB := (d_iQ_i^*)_{i\in I}$. By Lemma \ref{l:o-v-dual_charac}, there exists some $L\in\calB(\calH,\frakH)$ with $L^*T_\calW = 0$ such that $T_\calB = T_\calW S_\calW^{-1} + L$. Put $L_i := \frakE_i^*L$, $i\in I$. Then $\calL := (L_i)_{i\in I}$ is a $\calB(\calH)$-valued Bessel sequence with $T_\calL = L$. From $T_\calB = T_\calW S_\calW^{-1} + L$ we conclude that $d_iQ_i^* = c_iP_{W_i}S_\calW^{-1} + L_i$ for $i\in I$, that is, $d_iQ_i = c_iS_\calW^{-1}P_{W_i} + L_i^*$. And since $Q_i = Q_iP_{W_i}$, we obtain $d_iQ_i = (c_iS_\calW^{-1} + L_i^*)P_{W_i} = A_i$, $i\in I$. Therefore, $\ran A_i\subset V_i$ and $\|A_i\| = d_i$ if $i\notin I_0(\calV,\calW)$.

(iii)$\Sra$(i). For $i\in I$, define $Q_i := d_i^{-1}A_i$ if $d_i\neq 0$ and $Q_i := 0$ otherwise. Then $\ran Q_i\subset V_i$ and $W_i^\perp\subset\ker Q_i$, $i\in I$. If $d_i\neq 0$ then $\|Q_i\| = 1$ for $i\notin I_0(\calV,\calW)$. If $d_i = 0$ then $V_i = \{0\}$, i.e., $i\in I_0(\calV,\calW)$. Hence, $(Q_i)_{i\in I}$ satisfies \eqref{e:require}. Moreover, for $x\in\calH$ we have
\begin{align*}
\sum_{i\in I}c_id_iQ_ix
&= \sum_{d_i\neq 0}c_iA_ix = \sum_{i\in I}c_iA_ix = \sum_{i\in I}\left(c_i^2S_\calW^{-1}P_{W_i}x + c_iL_i^*P_{W_i}x\right)\\
&= S_\calW^{-1}S_\calW x + T_\calL^*T_\calW x = x,
\end{align*}
and (i) follows.
\end{proof}

As desired in {\bf (D4)} and proved in Proposition \ref{p:desiderata}, given a fusion frame $\calW=((W_i,c_i))_{i\in I}$, the Bessel fusion sequence $((S_\calW^{-1}W_i,c_i\|S_\calW^{-1}|W_i\|))_{i\in I}$ is always a FF-dual of $\calW$ (and therefore itself a fusion frame). In analogy with the vector frame setting, we call
\begin{equation}\label{e:canff}
\widetilde{\calW} := \left((S_\calW^{-1}W_i, c_i\|S_\calW^{-1}|W_i\|)\right)_{i\in I}
\end{equation}
the {\it canonical} fusion frame dual of $\calW=((W_i,c_i))_{i\in I}$. If $\ran T_\calW = W := \bigoplus_{i\in I}W_i$ then there exist no other FF-duals than extensions of the canonical one. Indeed, in this case, $T_\calL^*|W = 0$ in condition (iii) of Theorem \ref{t:charac_dual_fus}, and so $L_i^*|W_i = 0$ for each $i\in I$. It is an open question whether in all remaining cases there always exist FF-duals other than the canonical one or extensions of it. In \cite{hmbz} this was shown to be true for a very special class of harmonic fusion frames. In the finite-dimensional situation, the question can be answered in the affirmative, as proven in \cite[Prop. 3.9]{hm}. This fact can be also deduced from the next corollary, which is an immediate consequence of Theorem \ref{t:charac_dual_fus}.

\begin{cor}
Let $I$ be finite, and let $\calW = ((W_i,c_i))_{i\in I}$ be a fusion frame for the finite-dimensional Hilbert space $\calH$. Then for each $L\in\calB(\calH,\frakH)$ with $\ran L\subset\ker T_\calW^*$, the sequence $((\ran A_i,\|A_i\|))_{i\in I}$, where
$$
A_i := (c_iS_\calW^{-1} + L^*\frakE_i)P_{W_i},\quad i\in I,
$$
is a fusion frame dual of $\calW$.
\end{cor}

\section{Perturbations of Operator-valued Frame Sequences}\label{sec:Perturbations}
In this section, we prove that -- in some sense and under certain conditions -- duals of operator-valued  frames and frame sequences are stable under small perturbations. In our results we utilize the following notion of $\mu$-perturbation.
%generalize the notion of $\mu$-perturbation from \cite{hmp} to operator-valued Bessel sequences.

\begin{defn}\label{def:perturbation}
Let $\mu > 0$, and let $\calA$ and $\calB$ be two $\calB(\calH,\calK)$-valued Bessel sequences. We say that $\calB$ is a {\em $\mu$-perturbation} of $\calA$ {\rm (}and vice versa{\rm )} if
$$
\|T_{\calA} - T_{\calB}\|\,\le\,\mu.
$$
\end{defn}

\begin{rem}
{\rm (a)} We mention that the notion of $\mu$-perturbation is a special case of the perturbations of Paley-Wiener type which have been considered in, e.g., \cite[Thm. 15.1.1]{c} or in \cite{CC97}. 

\smallskip
{\rm (b)} The term $\mu$-perturbation as such was originally introduced in \cite{hmp} for vector sequences which might not be Bessel sequences. However, if two vector Bessel sequences are $\mu$-perturbations of one another in the sense of \cite{hmp}, this means that the difference of their synthesis operators has a norm which does not exceed $\mu$. Therefore, in the case $\calK = \C$, the above definition coincides with the one in \cite{hmp} \braces{for Bessel sequences}. Furthermore, we mention that $\|T_\calA - T_\calB\|\le\mu$ implies $\|A_i - B_i\|\le\mu$ for every $i\in I$ since
$$
\|A_i - B_i\| = \|\frakE_i^*T_\calA - \frakE_i^*T_\calB\|\le\|T_\calA - T_\calB\|\le\mu.
$$

\smallskip
{\rm (c)} Note that $(\calA,\calB)\mapsto\|T_\calA - T_\calB\| = \|T_{\calA - \calB}\|$ is the distance induced by the norm $\|\calC\| = \|T_\calC\|$ on $\frakB(I,\calH,\calK)$, see page \pageref{page:Besselnorm}.
\end{rem}

In order to treat perturbations of operator-valued frame sequences, we will make use of the notion of the gap between subspaces of a Hilbert space.

\subsection{Perturbations and the Gap Between Subspaces}
For two closed subspaces $V$ and $W$ of $\calH$ the {\em gap from $V$ to $W$} is defined by
$$
\delta(V,W) := \sup\left\{\|v - P_Wv\| : v\in V,\,\|v\|=1\right\} = \left\|(I - P_W)|V\right\| = \|P_{W^\perp}|V\|.
$$
We remark that in \cite{hmp}, $\delta(V,W)$ was called the gap {\em between} $V$ and $W$. Here, we agree to follow, e.g., \cite{ckkl}, and choose a different term in order to emphasize the order of $V$ and $W$ in $\delta(V,W)$. It is worth noting that
\begin{equation}\label{e:perp}
\delta(W^\perp,V^\perp) = \|P_V|W^\perp\| = \|(P_{W^\perp}|V)^*\| = \|P_{W^\perp}|V\| = \delta(V,W).
\end{equation}
Instead of the gap, some authors prefer to work with the {\em infimum cosine angle} $R(V,W)$ from $V$ to $W$ which is given by
$$
R(V,W) := \inf\left\{\|P_Wv\| : v\in V,\,\|v\|=1\right\}.
$$
It is easy to see that
\begin{equation}\label{e:transfer}
\delta(V,W) = \sqrt{1 - R(V,W)^2}.
\end{equation}
As $\delta$ is not a metric, Kato (see \cite[\paragraf IV.2]{k}) defines the {\em gap between $V$ and $W$} by
\begin{equation}\label{e:maxgap}
\Delta(V,W) := \max\left\{\delta(V,W),\delta(W,V)\right\},
\end{equation}
and shows that
$$
\Delta(V,W) = \|P_V - P_W\|.
$$
The next lemma is well known (see, e.g., \cite{bhkl,ckkl} or \cite[Theorem I-6.34]{k}). For the sake of completeness, we provide a short proof here.

\begin{lem}\label{l:gap}
The following statements hold.
\begin{enumerate}
\item[{\rm (i)}]  If $\delta(V,W) < 1$, then $V\cap W^\perp = \{0\}$, and $P_W|V\in\calB(V,W)$ is bounded below.
\item[{\rm (ii)}] If $\Delta(V,W) < 1$ then $\delta(V,W) = \delta(W,V)$, and the operators $P_W|V\in\calB(V,W)$ and $P_V|W\in\calB(W,V)$ are isomorphisms.
\end{enumerate}
\end{lem}
\begin{proof}
(i). From \eqref{e:transfer}, we see that $R(V,W) > 0$ which implies that the operator $P_W|V$ is bounded below. In particular, $V\cap W^\perp = \ker(P_W|V) = \{0\}$.

(ii). By (i), $P_W|V$ and $P_V|W$ are bounded below. And since $P_V|W = (P_W|V)^*$, we conclude from Lemma \ref{l:op1} that these operators are bijective. Finally,
$$
R(V,W)^{-1} = \|(P_W|V)^{-1}\| = \|((P_W|V)^*)^{-1}\| = \|(P_V|W)^{-1}\| = R(W,V)^{-1}
$$
proves $\delta(V,W) = \delta(W,V)$.
\end{proof}

The following well known lemma can be traced back to the paper \cite{m}. However, for the sake of self-containedness we provide a proof here. It is based on the fact that $\delta(V,W) = \sup\{\dist(v,W) : v\in V,\,\|v\|=1\}$ which easily follows from the definition.

\begin{lem}\label{l:gappyhelpy}
Let $X,Y$ be Hilbert spaces, $T,S\in\calB(X,Y)$, and assume that there exists $c > 0$ such that $\|Tx\|\ge c\|x\|$ for each $x\in(\ker T)^\perp$. Then $\ran T$ is closed and
$$
\delta\left(\ran T,\ol{\ran S}\right)\,\le\,\frac{\|T - S\|}c.
$$
\end{lem}
\begin{proof}
Using the formula from the above discussion, we have
\begin{align*}
\delta\left(\ran T,\ol{\ran S}\right)
&= \sup\left\{\inf_{z\in\ol{\ran S}}\|y - z\| : y\in\ran T,\,\|y\|=1\right\}\\
&= \sup\left\{\inf_{u\in X}\|Tx - Su\| : x\in(\ker T)^\perp,\,\|Tx\|=1\right\}\\
&\le \sup\left\{\|Tx - Sx\| : x\in(\ker T)^\perp,\,\|Tx\|=1\right\}\\
&\le \|T - S\|\,\sup\{\|x\| : x\in(\ker T)^\perp,\,\|Tx\|=1\}\le c^{-1}\|T - S\|.
\end{align*}
The fact that $\ran T$ is closed follows from Lemma \ref{l:op1}.
\end{proof}

It is clear that Lemma \ref{l:gappyhelpy} is only useful if $\|T-S\| < c$, that is, when $S$ is sufficiently close to $T$ in norm. In fact, since later $T$ and $S$ will be analysis operators of frames which are (small) perturbations of one another, the condition $\|T-S\| < c$ will be satisfied with $c$ being the perturbation parameter. The same remark also applies to the estimates in the next corollary.

\begin{cor}\label{c:delta_fs}
Let $\calA$ be a $\calB(\calH,\calK)$-valued frame sequence with lower frame bound $\alpha$ and let $\calB$ a $\calB(\calH,\calK)$-valued Bessel sequence. Then
$$
\delta(\calH_\calA,\calH_\calB)\,\le\,\frac{\|T_\calA - T_\calB\|}{\sqrt{\alpha}}
\qquad\text{and also}\qquad
\delta\left(\ran T_\calA,\ol{\ran T_\calB}\right)\,\le\,\frac{\|T_\calA - T_\calB\|}{\sqrt{\alpha}}.
$$
\end{cor}
\begin{proof}
The second relation follows immediately from Lemma \ref{l:gappyhelpy} with $X = \calH$, $Y = \frakK$, $T = T_\calA$, and $S = T_\calB$. For the first relation we choose $X = \frakK$, $Y = \calH$, $T = T_\calA^*$, and $S = T_\calB^*$. Then Lemma \ref{l:gappyhelpy} is applicable since $(\ker T_\calA^*)^\perp = \ran T_\calA$ and for $z\in\ran T_\calA$ we have $\|T_\calA^*z\|\ge\sqrt\alpha\|z\|$. The claim then follows from Lemma \ref{l:ranges}.
\end{proof}

The following theorem shows in particular that the perturbation of a frame sequence remains being a frame sequence when the perturbation parameter and the gap are sufficiently small. This was already proven in \cite[Thm. 2.1]{ckkl}. However, for the convenience of the reader, we present a short proof.

\begin{thm}\label{t:pert1}
Let $\calA$ be a $\calB(\calH,\calK)$-valued frame sequence with frame bounds $\alpha\le\beta$, and let $\calB$ be a $\mu$-perturbation of $\calA$. Then we have
\begin{equation}\label{e:first_delta_ov}
\delta(\calH_\calA,\calH_\calB)\,\le\,\frac{\mu}{\sqrt \alpha}.
\end{equation}
Moreover, if $\mu < \sqrt \alpha$ and $\Delta(\calH_\calA,\calH_\calB) < 1$, then the following statements hold:
\begin{itemize}
\item[{\rm (i)}]  $\calB$ is an operator-valued frame sequence with frame bounds
\begin{equation}\label{e:framebounds}
\big(\sqrt \alpha - \mu\big)^2
\quad\text{and}\quad
\big(\delta(\calH_\calB,\calH_\calA^\perp)\sqrt{\beta} + \mu\big)^2.
\end{equation}
If $\calA$ is an operator-valued frame for $\calH$ then so ist $\calB$.
\item[{\rm (ii)}] For the gap $\Delta(\ran T_\calA,\ran T_\calB)$ between the closed subspaces $\ran T_\calA$ and $\ran T_\calB$ we have
\begin{equation*}
\Delta(\ran T_\calA,\ran T_\calB)\le\frac{\mu}{\sqrt \alpha - \mu}.
\end{equation*}
\end{itemize}
\end{thm}
\begin{proof}
The relation \eqref{e:first_delta_ov} follows directly from Corollary \ref{c:delta_fs}. Assume now that $\mu < \sqrt \alpha$ and $\Delta := \Delta(\calH_\calA,\calH_\calB) < 1$.

(i). For the upper frame bound of $\calB$, let $x\in\calH_\calB$. Then
$$
\|T_\calB x\|\le\|(T_\calB - T_\calA)x\| + \|T_\calA(P_\calA|\calH_\calB)x\|\le\left(\mu + \sqrt{\beta}\|P_\calA|\calH_\calB\|\right)\|x\|.
$$
For the lower  frame bound, let $x\in\calH_\calA$. Then we have (see \eqref{e:TH})
\begin{align*}
\|T_\calB P_\calB x\| = \|T_\calB x\|\ge\|T_\calA x\| - \|(T_\calA - T_\calB)x\|\ge(\sqrt \alpha - \mu)\|x\|\ge(\sqrt \alpha - \mu)\|P_\calB x\|.
\end{align*}
And as (due to $\Delta < 1$) $P_\calB$ maps $\calH_\calA$ bijectively onto $\calH_\calB$, a lower frame bound of $\calB$ is $(\sqrt \alpha - \mu)^2$. It also follows from $\Delta < 1$ that $\calB$ is an operator-valued frame if $\calA$ is.

(ii). Applying Corollary \ref{c:delta_fs} to both $(\calA,\calB)$ and $(\calB,\calA)$, we obtain
$$
\delta(\ran T_\calA,\ran T_\calB)\le\frac{\mu}{\sqrt\alpha}\qquad\text{and}\qquad\delta(\ran T_\calB,\ran T_\calA)\le\frac{\mu}{\sqrt\alpha - \mu}.
$$
Since the second right hand side is larger than the first (unless $\mu = 0$), the claim follows from \eqref{e:maxgap}.
\end{proof}

\begin{rem}\label{r:ab}
{\rm (a)} The condition $\Delta(\calH_\calA,\calH_\calB) < 1$ cannot be omitted in {\rm (i)}. A counterexample in the vector sequence case can be found in \cite[Example 15.3.1]{c}.

\smallskip
{\rm (b)} Note that $\Delta$ in Theorem \rmref{t:pert1} tends {\rm (}linearly{\rm )} to zero with $\mu$ {\rm (}cf.\ \eqref{e:first_delta_ov}{\rm )}.
\end{rem}

\subsection{The Perturbation Effect on the Duals}
In \cite{hmp}, it was studied how perturbations of a vector frame sequence affect the canonical dual. A similar approach was made in \cite{s07} for $G$-frames. In the following, we consider the problem in a much more general setting. Firstly, we consider operator-valued frame sequences and secondly, we include all duals in our considerations. Since our methods are different from those in \cite{hmp} and \cite{s07}, we can significantly improve the estimates.

For the formulation of the following statements we ask the reader to recall the parametrization $\calL_\calA\to\calD(\calA)$, $L\mapsto\wt\calA(L)$, in \eqref{eq:A_tilde} of the duals of an operator-valued frame sequence $\calA$.

\begin{lem}\label{l:dual_op_diff}
Let $\calA$ and $\calB$ be two $\calB(\calH,\calK)$-valued frame sequences and let $L\in\calL_\calA$ and $M\in\calL_\calB$. Then we have
\begin{align}
\begin{split}\label{e:darst}
T_{\wt\calB(M)} - T_{\wt\calA(L)}
&= T_\calB(S_\calB|\calH_\calB)^{-1}P_\calB(T_\calA^* - T_\calB^*)RP_\calA P_\calB\\
&\hspace{.92cm} + T_\calB(S_\calB|\calH_\calB)^{-1}P_\calB(I - P_\calA)P_\calB\\
&\hspace{3.1cm} - RP_\calA(I - P_\calB)\\
&\hspace{1.82cm} + \left(M - P_{\ker T_\calB^*}RP_\calA\right)P_\calB,
\end{split}
\end{align}
where $R := T_\calA(S_\calA|\calH_\calA)^{-1} + L\in\calB(\calH_\calA,\frakK)$.
\end{lem}
\begin{proof}
For a more comfortable reading we write $S_\calB^{-1}$ instead of $(S_\calB|\calH_\calB)^{-1}$ (analogously for $\calA$). Consider the right hand side of \eqref{e:darst}. Let us first extract the sum of those terms which do not contain $R$:
$$
S_1 := T_\calB S_\calB^{-1}P_\calB(I - P_\calA)P_\calB + MP_\calB = T_{\wt\calB(M)} - T_\calB S_\calB^{-1}P_\calB P_\calA P_\calB.
$$
Now we extract the sum of terms containing $L$ (inside $R$). Note that $T_\calA^*L = 0$ and $P_{\ran T_\calB} = T_\calB(S_\calB|\calH_\calB)^{-1}T_\calB^*$. This sum is
\begin{align*}
S_2 :&= T_\calB S_\calB^{-1}P_\calB(T_\calA^* - T_\calB^*)LP_\calA P_\calB - LP_\calA(I - P_\calB) - P_{\ker T_\calB^*}LP_\calA P_\calB\\
&= -T_\calB S_\calB^{-1}T_\calB^*LP_\calA P_\calB - LP_\calA(I - P_\calB) - P_{\ker T_\calB^*}LP_\calA P_\calB\\
&= -LP_\calA P_\calB - LP_\calA(I - P_\calB) = -LP_\calA.
\end{align*}
The rest of the sum is given by
\begin{align*}
S_3 :&= T_\calB S_\calB^{-1}P_\calB(T_\calA^* - T_\calB^*)T_\calA S_\calA^{-1}P_\calA P_\calB - T_\calA S_\calA^{-1}P_\calA(I - P_\calB) - P_{\ker T_\calB^*}T_\calA S_\calA^{-1}P_\calA P_\calB\\
&= (T_\calB S_\calB^{-1}P_\calB - P_{\ran T_\calB}T_\calA S_\calA^{-1})P_\calA P_\calB - T_\calA S_\calA^{-1}P_\calA(I - P_\calB) - P_{\ker T_\calB^*}T_\calA S_\calA^{-1}P_\calA P_\calB\\
&= T_\calB S_\calB^{-1}P_\calB P_\calA P_\calB - T_\calA S_\calA^{-1}P_\calA P_\calB - T_\calA S_\calA^{-1}P_\calA(I - P_\calB)\\
&= T_\calB S_\calB^{-1}P_\calB P_\calA P_\calB - T_\calA S_\calA^{-1}P_\calA.
\end{align*}
Thus,
$$
S_1 + S_2 + S_3 = T_{\wt\calB(M)} - LP_\calA - T_\calA S_\calA^{-1}P_\calA = T_{\wt\calB(M)} - T_{\wt\calA(L)},
$$
and the lemma is proven.
\end{proof}

Let us first study how perturbation effects the canonical duals of original and perturbed sequence.

\begin{thm}\label{t:canonicals}
Let $\calA$ be a $\calB(\calH,\calK)$-valued frame sequence with lower frame bound $\alpha$, let $\calB$ be a $\mu$-perturbation of $\calA$, $\mu < \sqrt \alpha$, and assume that the gap $\Delta := \Delta(\calH_\calA,\calH_\calB) < 1$. Then $\calB$ is a $\calB(\calH,\calK)$-valued frame sequence and for the canonical duals $\wt\calA$ and $\wt\calB$ of $\calA$ and $\calB$, respectively, we have
$$
\left\|T_{\wt\calA} - T_{\wt\calB}\right\|\,\le\,\frac{2\mu + (2\sqrt\alpha - \mu)\Delta}{\sqrt\alpha(\sqrt\alpha - \mu)}.
$$
\end{thm}
\begin{proof}
The fact that $\calB$ is a $\calB(\calH,\calK)$-valued frame sequence follows directly from Theorem \ref{t:pert1}. This theorem also yields that $\calB$ has the lower frame bound $(\sqrt\alpha - \mu)^{2}$ and that $\Delta(\ran T_\calA,\ran T_\calB)\le\tfrac{\mu}{\sqrt\alpha - \mu}$. We now make use of Lemma \ref{l:dual_op_diff} and find that (setting $M = 0$ and $L = 0$)
\begin{align*}
T_{\wt\calB} - T_{\wt\calA}
&= T_\calB S_\calB^{-1}P_\calB(T_\calA^* - T_\calB^*)T_\calA S_\calA^{-1}P_\calA P_\calB + T_\calB S_\calB^{-1}P_\calB(I - P_\calA)P_\calB\\
&\hspace{2.26cm}- T_\calA S_\calA^{-1}P_\calA(I - P_\calB) - P_{\ker T_\calB^*}T_\calA S_\calA^{-1}P_\calA P_\calB\\
&= T_\calB S_\calB^{-1}P_\calB(T_\calA^* - T_\calB^*)T_\calA S_\calA^{-1}P_\calA P_\calB + (P_{\ker T_\calA^*} - P_{\ker T_\calB^*})T_\calA S_\calA^{-1}P_\calA P_\calB\\
&\hspace{3.3cm}+ T_\calB S_\calB^{-1}P_\calB(I - P_\calA)P_\calB - T_\calA S_\calA^{-1}P_\calA(I - P_\calB),
\end{align*}
where we again agree to write $S_\calA^{-1} = (S_\calA|\calH_\calA)^{-1}$ and $S_\calB^{-1} = (S_\calB|\calH_\calB)^{-1}$. Now, we observe that
\[
\|(I - P_\calA)P_\calB\| = \|(I - P_\calA)|\calH_\calB\| = \Delta = \|(I - P_\calB)|\calH_\calA\| = \|(I - P_\calB)P_\calA\|.
\]
Making extensive use of \eqref{e:la_tormenta}, we obtain
$$
\left\|T_{\wt\calB} - T_{\wt\calA}\right\|
\le \frac{\mu}{\sqrt\alpha(\sqrt\alpha - \mu)} + \frac{\Delta(\ker T_\calA^*,\ker T_\calB^*)}{\sqrt\alpha} + \left(\frac{1}{\sqrt\alpha - \mu} + \frac{1}{\sqrt\alpha}\right)\Delta.
$$
But this is the claim as $\Delta(\ker T_\calA^*,\ker T_\calB^*) = \Delta(\ran T_\calA,\ran T_\calB)$, see \eqref{e:perp}.
\end{proof}

Let us state Theorem \ref{t:canonicals} especially for the case of (operator-valued) frames (where $\calH_\calA = \calH_\calB = \calH$ and $\Delta = 0$).

\begin{cor}\label{c:canonicals_frame}
Let $\calA$ be a $\calB(\calH,\calK)$-valued frame for $\calH$ with lower frame bound $\alpha$ and let $\calB$ be a $\mu$-perturbation of $\calA$, $\mu < \sqrt\alpha$. Then $\calB$ is a $\calB(\calH,\calK)$-valued frame for $\calH$ and for the canonical duals $\wt\calA$ and $\wt\calB$ of $\calA$ and $\calB$, respectively, we have
$$
\left\|T_{\wt\calA} - T_{\wt\calB}\right\|\,\le\,\frac{2\mu}{\sqrt\alpha(\sqrt\alpha - \mu)}.
$$
\end{cor}

\begin{rem}
As mentioned above, the perturbation effect on the canonical duals has been considered before in \cite{hmp} and \cite{s07}. However, our setting is more general and the estimates are better. For example, \cite[Theorem 4.1]{s07} states that if $\calA_1$ and $\calA_2$ are $\calB(\calH,\calK)$-valued frames with frame bounds $\alpha_j\le\beta_j$ {\rm (}$j=1,2${\rm )} which are $\mu$-perturbations of each other, then
$$
\|T_{\wt\calA_1} - T_{\wt\calA_2}\|\le\frac{\alpha_1 + \beta_1 + \sqrt{\beta_1\beta_2}}{\alpha_1\alpha_2}\,\mu.
$$
With $\alpha_2 = (\sqrt{\alpha_1} - \mu)^2$ and $\beta_2 = (\sqrt{\beta_1} + \mu)^2$ {\rm (}cf.\ \eqref{e:framebounds}{\rm )}, this is
$$
\|T_{\wt\calA_1} - T_{\wt\calA_2}\|\le\frac{(\alpha_1 + 2\beta_1 + \sqrt{\beta_1}\mu)\mu}{\alpha_1(\sqrt{\alpha_1} - \mu)^2}.
$$
It is now not hard to see that the bound in Corollary \rmref{c:canonicals_frame} is significantly better.
\end{rem}

In the following, we shall study the perturbation effect on the alternate duals. In particular, we show that whenever $\calA$ is an operator-valued frame sequence, $\wt\calA$ a dual of $\calA$, and $\calB$ a small perturbation of $\calA$, then there is a dual $\wt\calB$ of $\calB$ which is also a small perturbation of $\wt\calA$. We explicitly specify this dual.

\begin{thm}\label{t:o-v-duals_sequences}
Let $\calA$ be a $\calB(\calH, \calK)$-valued frame sequence with lower frame bound $\alpha$ and let $\calB$ be a $\calB(\calH,\calK)$-valued $\mu$-perturbation of $\calA$, $\mu < \sqrt{\alpha}$, such that $\Delta := \Delta(\calH_\calA,\calH_\calB) < 1$. Then $\calB$ is an operator-valued frame sequence, and for every $L\in\calL_\calA$ the dual
$$
\wt\calB_L := \wt\calB\left(P_{\ker T_\calB^*}T_{\wt\calA(L)}|\calH_\calB\right)
$$
of $\calB$ is a $\la$-perturbation of $\wt\calA(L)$, where
\begin{equation}\label{e:o-v-lambda}
\la = \frac{\mu + (2\sqrt\alpha - \mu)\Delta}{\sqrt\alpha (\sqrt\alpha - \mu)} + \left(\frac{\mu}{\sqrt\alpha - \mu} + \Delta\right)\|L\|.
\end{equation}
\end{thm}
\begin{proof}
First note that by Lemma \ref{l:o-v-dual_charac}, $\wt\calB_L$ is a dual of $\calB$. 
Using the notation from Lemma \ref{l:dual_op_diff}, we have $T_{\wt\calA(L)} = RP_\calA$. Setting $M := P_{\ker T_\calB^*}T_{\wt\calA(L)}|\calH_\calB$, we observe that
\begin{align*}
T_{\wt\calB_L} - T_{\wt\calA(L)}
&= T_\calB(S_\calB|\calH_\calB)^{-1}P_\calB(T_\calA^* - T_\calB^*)RP_\calA P_\calB\\
&\hspace{.4cm}+ T_\calB(S_\calB|\calH_\calB)^{-1}P_\calB(I - P_\calA)P_\calB\\
&\hspace{.4cm}- RP_\calA(I - P_\calB).
\end{align*}
Hence,
$$
\left\|T_{\wt\calB_L} - T_{\wt\calA(L)}\right\|\,\le\,\frac{\mu\|R\|}{\sqrt\alpha - \mu} + \frac{\Delta}{\sqrt\alpha - \mu} + \|R\|\Delta\,\le\,\la,
$$
since $\|R\|\le\alpha^{-1/2} + \|L\|$.
\end{proof}

In the case of operator-valued frames, Theorem \ref{t:o-v-duals_sequences} reduces to the first statement of the next theorem. In this case our choice of the dual of the perturbed frame turns out to be perfect in terms of best approximations.

\begin{thm}\label{t:pert_frame_duals}
Let $\calA$ be a $\calB(\calH,\calK)$-valued frame with lower frame bound $\alpha$ and let $\calB$ be a $\mu$-perturbation of $\calA$, where $\mu < \sqrt{\alpha}$. Then $\calB$ is an operator-valued  frame, and for each $L\in\calL_\calA$ the dual $\wt\calB_L := \wt\calB(P_{\ker T_\calB^*}T_{\wt\calA(L)})$ of $\calB$ is a best approximation\footnote{With respect to the norm $\|\calA\| = \|T_\calA\|$ on $\frakB(I,\calH,\calK)$, see page \pageref{page:Besselnorm}.} of $\wt\calA(L)$ in $\calD(\calB)$ and a $\la$-perturbation of $\wt\calA(L)$, where
\begin{equation}\label{e:lambda2}
\la = \frac{\mu}{\sqrt\alpha - \mu}\left(\frac{1}{\sqrt\alpha} + \|L\|\right).
\end{equation}
\end{thm}
\begin{proof}
We only have to show that $\wt\calB_L$ is a best approximation of $\wt\calA(L)$ in $\calD(\calB)$. Note that $P_\calA = P_\calB = I$ in this case. We have to show that $\|T_{\wt\calB_L} - T_{\wt\calA(L)}\|\le\|T_{\wt\calB(M)} - T_{\wt\calA(L)}\|$ for all $M\in\calL_\calB$. Put $M_0 := P_{\ker T_\calB^*}R$, where $R = T_\calA S_\calA^{-1} + L$. Then $\wt\calB_L = \wt\calB(M_0)$ and for arbitrary $M\in\calL_\calB$ we have
\begin{align*}
T_{\wt\calB(M)} - T_{\wt\calA(L)} = \left(T_\calB S_\calB^{-1} + M\right) - R = P_{\ran T_\calB}\left(T_\calB S_\calB^{-1} - R\right) + P_{\ker T_\calB^*}(M - R).
\end{align*}
In particular, $T_{\wt\calB(M_0)} - T_{\wt\calA(L)} = P_{\ran T_\calB}\left(T_\calB S_\calB^{-1} - R\right)$. Thus, for any $x\in\calH$ we have
$$
\left\|\left(T_{\wt\calB(M)} - T_{\wt\calA(L)}\right)x\right\|^2\,\ge\,\left\|P_{\ran T_\calB}\left(T_\calB S_\calB^{-1} - R\right)x\right\|^2 = \left\|\left(T_{\wt\calB_L} - T_{\wt\calA(L)}\right)x\right\|^2,
$$
which proves the claim.
\end{proof}

\begin{rem}\label{r:many_things}
{\rm (a)} Intriguingly, Theorem \rmref{t:o-v-duals_sequences} shows that the canonical dual $\wt\calB$ of $\calB$ is in general not a best approximation of the canonical dual $\wt\calA$ of $\calA$ in $\calD(\calB)$. Indeed, in the operator-valued frame case it follows from the above proof that for $x\in\calH$ \braces{and $L=0$} we have
\begin{align*}
\left\|\left(T_{\wt\calB} - T_{\wt\calA}\right)x\right\|^2 - \left\|\left(T_{\wt\calB_L} - T_{\wt\calA}\right)x\right\|^2 = \left\|P_{\ker T_\calB^*}T_\calA S_\calA^{-1}x\right\|^2.
\end{align*}
For example, assume that $I$ is finite, $\dim\calH < \infty$, and $\ran T_\calA\cap\ran T_\calB = \{0\}$. Then $P_{\ker T_\calB^*}T_\calA S_\calA^{-1}$ is injective. Thus, there exists $c > 0$ such that $\|P_{\ker T_\calB^*}T_\calA S_\calA^{-1}x\|\ge c\|x\|$ for all $x\in\calH$. This yields
$$
\|T_{\wt\calB} - T_{\wt\calA}\|^2\ge \|T_{\wt\calB_L} - T_{\wt\calA}\|^2 + c^2 > \|T_{\wt\calB_L} - T_{\wt\calA}\|^2.
$$
Hence, in this case, $\wt\calB$ is not a best approximation of $\wt\calA$ in $\calD(\calB)$.

\smallskip
{\rm (b)} Let $\calH$ be finite-dimensional and $|I|$ be finite. Then $(\calB(\calH,\frakK),\<\cdot\,,\cdot\>_{\rm HS})$ is a Hilbert space, $T_\calB S_\calB^{-1}$ is perpendicular to $\calL_\calB$, and the orthogonal projection onto $\calL_\calB$ is given by \eqref{e:op_fd} (see Remark \ref{r:HS}). Hence, the orthogonal projection onto the affine subspace $\calD(\calB) = T_\calB S_\calB^{-1} + \calL_\calB$ is given by
$$
P_{\calD(\calB)}X = T_\calB S_\calB^{-1} + P_{\ker T_\calB^*}X,\qquad X\in\calB(\calH,\frakK).
$$
Therefore, {\it the} best approximation in the Hilbert space $(\calB(\calH,\frakK),\<\cdot\,,\cdot\>_{\rm HS})$ of $T_{\wt A(L)} = T_\calA S_\calA^{-1} + L$ in $\calD(\calB)$ is
$$
P_{\calD(\calB)}T_{\wt A(L)} = T_\calB S_\calB^{-1} + P_{\ker T_\calB^*}T_{\wt A(L)},
$$
which is exactly the analysis operator belonging to the dual of $\calB$ that we have chosen in Theorem \ref{t:pert_frame_duals}.

\smallskip
{\rm (c)} Note that in all estimates above the upper frame bound of $\calA$ does not play a role.

\smallskip
{\rm (d)} During our studies on the subject, we came across the paper \cite{ag13} where the authors prove the following: {\em Given an operator-valued frame $\calA$, a dual $\wt\calA$ of $\calA$ and $\mu > 0$ sufficiently small, then for every operator-valued frame $\calB$ being a $\mu$-perturbation of $\calA$ there exist $C > 0$ and a dual $\wt\calB$ of $\calB$ which is a $C\mu$-perturbation of $\wt\calA$}. The main differences to {\rm Theorem \rmref{t:pert_frame_duals}} are as follows:
\begin{itemize}
	\item In \cite{ag13} it is required that $\mu < \|L\|^{-1}$, whereas in {\rm Theorem \ref{t:pert_frame_duals}}, $\mu$ and $L$ are unrelated and $L$ can be arbitrary.
	\item We explicitly provide the dual $\wt\calB$ of $\calB$ whereas in \cite{ag13} it only appears in the proof: it is defined via $T_{\wt\calB} := (T_\calB S_\calB^{-1} + L)(I + T_\calB^*L)^{-1}$ \braces{which exists as $\mu < \|L\|^{-1}$ is assumed and $T_\calA^*L = 0$}.
	\item We evaluate the quality of our choice of the dual $\wt\calB$ by proving that it is a best approximation of $\wt\calA(L)$ in $\calD(\calB)$. Such an analysis is not contained in \cite{ag13}.
	\item The constant $C$ in \cite{ag13} is hidden in the proof and depends on the frame bounds of $\calA$, $\calB$, and $\wt\calA$, and on $\|L\|$. Here, we explicitly specify $C = \la$ which does only depend on the lower frame bound of $\calA$ and $\|L\|$.
\end{itemize}
Hence, {\rm Theorem \rmref{t:pert_frame_duals}} can be seen as an essential improvement of \cite{ag13}.
\end{rem}

Theorem \ref{t:o-v-duals_sequences} shows that, given two operator-valued frame sequences $\calA$ and $\calB$ which are close to each other, then for any dual $\wt\calA$ of $\calA$ there exists a special dual $\wt\calB$ of $\calB$ which is close to $\wt\calA$ (in the frame situation as close as possible). The following proposition now states that the mapping $\wt\calA\mapsto\wt\calB$ between $\calD(\calA)$ and $\calD(\calB)$ is one-to-one and onto.

\begin{prop}
Let $\calA$ be a $\calB(\calH,\calK)$-valued frame sequence with lower frame bound $\alpha$ and let $\calB$ be a $\mu$-perturbation of $\calA$ such that $\Delta(\calH_\calA,\calH_\calB) < 1$ and $\mu<\sqrt{\alpha}/2$. Then the {\rm (}affine{\rm )} mapping
$$
\calD(\calA)\to\calD(\calB),\qquad \wt\calA(L)\mapsto\wt\calB\left(P_{\ker T_\calB^*}T_{\wt\calA(L)}|\calH_\calB\right),
$$
is bijective.
\end{prop}
\begin{proof}
It obviously suffices to prove that the affine mapping $\calL_{\mathcal A}\to\calL_{\mathcal B}$, $L\mapsto P_{\ker T_\calB^*}T_{\wt\calA(L)}|\calH_\calB$, is bijective. Removing the constant affine part reduces the task to showing that the linear map
$$
\calR : \calL_\calA\to\calL_\calB,\qquad \calR X := P_{\ker T_\calB^*}X(P_\calA|\calH_\calB),\quad X\in\calL_\calA,
$$
is bijective. For this, we observe that $\mu < \sqrt\alpha/2$ and Theorem \ref{t:pert1} (ii) imply that $\Delta(\ran T_\calA,\ran T_\calB) < 1$. Thus, also $\Delta(\ker T_\calA^*,\ker T_\calB^*) < 1$ (see \eqref{e:perp}). Define an operator $\calQ$ by
$$
\calQ : \calL_\calB\to\calL_\calA,\qquad \calQ Y := \left(P_{\ker T_\calB^*}|\ker T_\calA^*\right)^{-1}Y\left(P_\calA|\calH_\calB\right)^{-1},\quad Y\in\calL_\calB.
$$
Note that the inverses in the definiton of $\calQ$ exist due to Lemma \ref{l:gap}(ii) and that $Y$ maps $\calH_\calB$ to $\ker T_\calB^*$ so that $\calQ$ is well defined. Now, for $X\in\calL_\calA$ and $Y\in\calL_\calB$ we have
$$
\calQ\calR X = \left(P_{\ker T_\calB^*}|\ker T_\calA^*\right)^{-1}P_{\ker T_\calB^*}X(P_\calA|\calH_\calB)\left(P_\calA|\calH_\calB\right)^{-1} = X,
$$
as well as
$$
\calR\calQ Y = P_{\ker T_\calB^*}\left(P_{\ker T_\calB^*}|\ker T_\calA^*\right)^{-1}Y\left(P_\calA|\calH_\calB\right)^{-1}(P_\calA|\calH_\calB) = Y.
$$
Hence, $\calR^{-1} = \calQ$ exists.
\end{proof}

\begin{ex}
Let us consider the so-called Mercedes-Benz frame in $\R^2$, which is defined as
$$
\calA = \left\{\sqrt{\frac{2}{3}}(0,1),\sqrt{\frac{2}{3}}\left(\frac{\sqrt{3}}{2},-\frac{1}{2}\right),  
\sqrt{\frac{2}{3}}\left(-\frac{\sqrt{3}}{2},-\frac{1}{2}\right)\right\}.
$$
Let $\varepsilon\in (0,\sqrt{15}/4)$ and let $\calB$ be the set
$$
\calB = \left\{\sqrt{\frac{2}{3}}\left(\varepsilon,\sqrt{1-\varepsilon^2}\right),\sqrt{\frac{2}{3}}\left(\frac{\sqrt{3}}{2},-\frac{1}{2}\right),  
\sqrt{\frac{2}{3}}\left(-\frac{\sqrt{3}}{2},-\frac{1}{2}\right)\right\}.
$$
The analysis operators $T_\calA,T_\calB\in\calB(\R^2, \R^3)\cong\R^{3\times 2}$ are given by
$$
T_\calA = \frac 1 {\sqrt 6}\begin{pmatrix}
                0&2\\
								\sqrt{3}&-1\\
                -\sqrt{3}&-1
               \end{pmatrix}
\qquad\text{and}\qquad
T_\calB = \frac 1 {\sqrt 6}\begin{pmatrix}
                2\veps&2\sqrt{1-\varepsilon^2}\\
                \sqrt{3}&-1\\
                -\sqrt{3}&-1\\
               \end{pmatrix}
$$
and $\|T_\calA-T_\calB\| = 2\sqrt{\delta/3}$, where $\delta = 1 - \sqrt{1 - \veps^2}$. As $\veps < \sqrt{15}/4$, we have that $2\sqrt{\delta/3} < 1$, and since $\calA$ is a tight frame for $\R^2$ with frame bound $1$, it follows from Theorem \ref{t:pert1} that $\calB$ is a frame for $\R^2$ with frame bounds $(1-\varepsilon)^2$, $(1+\varepsilon)^2$. 

Given any dual of $\calA$, we will compute explicitly the dual of $\calB$ which is the best approximation in the sense of Theorem \ref{t:pert_frame_duals}. For this, recall that by Lemma \ref{l:o-v-dual_charac}, a dual $\wt\calA(L)$ of $\calA$ has the analysis operator $T_{\wt \calA(L)}=T_\calA S_\calA^{-1}+L$, where $L\in \R^{3\times 2}$ such that $L^*T_\calA=0$. An easy calculation shows that any $L$ with the latter property is of the form
$$
L=\begin{pmatrix}a&b\\a&b\\a&b\\\end{pmatrix},\qquad a,b\in\R.
$$
Let us fix $a,b\in\R$ and thereby a dual $\wt A(L)$ for $\calA$. Then, by Theorem \ref{t:pert_frame_duals}, the best approximation to $\wt A(L)$ in $\calD(\calB)$ is $\wt \calB_L=\wt B(P_{\ker T_\calB^*}T_{\wt\calA(L)})$ whose analysis operator is 
\begin{align*}
T_{\wt \calB_L}
&=T_\calB S_\calB^{-1}+P_{\ker T_\calB^*}T_{\wt\calA(L)} =  T_\calB S_\calB^{-1}+(I-T_\calB S_\calB^{-1}T_\calB^*)T_{\wt\calA(L)}\\
&= \frac{1}{\Delta}
\begin{pmatrix}
3(1 + 2t)a  &  \left(\sqrt{6} + 3b\right)(1 + 2t)\\
\frac\Delta{\sqrt 2} + p_-(t)a\;   &  -\sqrt{\frac 3 2}\Delta + \left(\sqrt{\frac 2 3} + b\right)p_-(t)\\
\frac{-\Delta}{\sqrt 2} + p_+(t)a\;   &   -\sqrt{\frac 3 2}\Delta + \left(\sqrt{\frac 2 3} + b\right)p_+(t)
\end{pmatrix},
\end{align*}
where $\Delta = 9-4\veps^2$, $t = \sqrt{1-\veps^2}$, and $p_\pm(t) = 6t^2 + (3\pm2\sqrt 3\veps)t\pm\sqrt 3\veps$.

As is easily seen, for each considered $\veps$ the best approximation of the canonical dual of $\calA$ ($L=0$) in $\calD(\calB)$ is not the canonical dual of $\calB$(cf.\ Remark \ref{r:many_things}). But let us compute the dual of $\calA$ having the canonical dual of $\calB$ as its best approximation in $\calD(\calB)$. One possibility to do this is to compute the first row of $T_\calB S_\calB^{-1}$ and then to compare with that of $T_{\wt\calB_L}$ above to obtain $a$ and $b$. The first row of $T_\calB S_\calB^{-1}$ is given by
$$
e_1^TT_\calB S_\calB^{-1} = \sqrt{2/3}(\veps,t)S_\calB^{-1} = \frac{\sqrt 6}\Delta\left(\veps\,,\,3t\right).
$$
From here, we retrieve $(a,b) = \sqrt{2/3}(1+2t)^{-1}(\veps,t-1)$.

\end{ex}

\section{Perturbations of Fusion Frames}\label{sec:Perturbations_FF}
This section is devoted to studying the behavior of the canonical FF-dual (as defined in \eqref{e:canff}) under perturbations. We consider perturbations of fusion frames in the same sense as of operator-valued frames. More precisely:

\begin{defn}\label{d:perturbation_fus}
Let $\mu > 0$, and let $\calW=((W_i,c_i))_{i\in I}$ and $\calV =((V_i,d_i))_{i\in I}$ be two Bessel fusion sequences in $\calH$. We say that $\calV$ is a {\em $\mu$-perturbation} of $\calW$ {\rm (}and vice versa{\rm )} if
$(d_iP_{V_i})_{i\in I}$ is a $\mu$-perturbation of $(c_iP_{W_i})_{i\in I}$ in the sense of {\rm Definition \ref{def:perturbation}}, that is, when 
$
\|T_{\calW} - T_{\calV}\|\,\le\,\mu.
$
\end{defn}

\begin{rem}\label{r:fusi}
{\rm (a)} If $\calV = ((V_i,d_i))_{i\in I}$ is a $\mu$-perturbation of $\calW = ((W_i,c_i))_{i\in I}$, then for each $i\in I$ we have
\begin{equation}\label{e:projections}
\|c_iP_{W_i} - d_iP_{V_i}\| = \|(T_\calW^* - T_\calV^*)\frakE_i\|\le\mu.
\end{equation}
In particular, if $c_i = d_i = 1$, $i\in I$, then $\Delta(W_i,V_i)\le\mu$ for all $i\in I$. Since $c_i = \|c_iP_{W_i}\|\,\le\,\|c_iP_{W_i} - d_iP_{V_i}\| + d_i$ and $d_i\le\|c_iP_{W_i} - d_iP_{V_i}\| + c_i$, relation \eqref{e:projections} implies
$$
|c_i - d_i|\,\le\,\|c_iP_{W_i} - d_iP_{V_i}\|\,\le\,\mu.
$$

\medskip\noindent
{\rm (b)} In \cite{ckl}, a different notion of perturbation for fusion frames was considered {\rm (}see \cite[Definition 5.1]{ckl}{\rm )} and it was proven that fusion frames are stable under these  perturbations {\rm (}see \cite[Proposition 5.8]{ckl}{\rm )}. When $\calW=((W_i,c_i))_{i\in I}$ and $\calV =((V_i,c_i))_{i\in I}$ are Bessel fusion sequences and the sequence of weights $c:=\{c_i\}_{i\in I}$ belongs to $\ell^2(I)$, the notion of perturbation in \cite{ckl} implies that $\calV$ is a $\mu$-perturbation of $\calV$ for $\mu=(\la_1+ \la_2+\veps)\|c\|_{\ell^2(I)}$, with $\la_1,\la_2\geq0$ and $\veps>0$ being the perturbation parameters of \cite[Definition 5.1]{ckl}. On the other hand, by \eqref{e:projections}, {\rm Definition \ref{d:perturbation_fus}} implies \cite[Definition 5.1]{ckl} only when the weights satisfy $\inf_{i\in I} c_i >0$. However, in the finite-dimensional setting, both notions of perturbation are equivalent.
\end{rem}

In the following, we will show that the canonical FF-dual of a $\mu$-perturbation of a fusion frame $\calW$ will be a $C\mu$-perturbation of the canonical FF-dual $\wt\calW$ of $\calW$, where $C > 0$ depends on $\mu$ and $\calW$. For this, we shall exploit the following two lemmas.

\begin{lem}\label{l:PQ}
Let $P$ and $Q$ be orthogonal projections in $\calH$ and $c,d > 0$. Then
$$
\|P - Q\|\,\le\,\sqrt{\frac 1 {c^2} + \frac 1 {d^2}}\cdot\|cP - dQ\|.
$$
\end{lem}
\begin{proof}
Let $x\in\calH$, $\|x\|=1$. Then we have
\begin{align*}
\|cPx - dQx\|^2
&= \|cQPx + c(I-Q)Px - dQx\|^2\\
&= \|Q(cPx - dx)\|^2 + c^2\|(I-Q)Px\|^2\\
&\ge c^2\|(I-Q)Px\|^2.
\end{align*}
Analogously, one obtains $\|cPx - dQx\|^2\ge d^2\|(I-P)Qx\|^2$. Thus, we have
$$
\|(I-Q)P\|\,\le\,\frac 1 c\|cP - dQ\|
\quad\text{and}\quad
\|(I - P)Q\|\,\le\,\frac 1 d\|cP - dQ\|.
$$
Hence, also $\|Q(I-P)\| = \|((I-P)Q)^*\| = \|(I-P)Q\|\le\tfrac 1 d\|cP - dQ\|$. Since, for $x\in\calH$,
$$
\|(P - Q)x\|^2 = \|QPx + (I-Q)Px - Qx\|^2 = \|Q(I-P)x\|^2 + \|(I-Q)Px\|^2,
$$
the claim follows from the above inequalities.
\end{proof}

\begin{lem}\label{l:R_W}
Let $W\subset\calH$ be a closed subspace and $A\in\calB(\calH)$ boundedly invertible. Then, for every $\la > 0$, the operator
$$
R(\la) := AP_W + \la A^{-*}P_{W^\perp}
$$
where $A^{-*} = (A^{-1})^*$, is boundedly invertible and we have
\begin{equation}\label{e:PS-1W}
P_{AW} = R(\la)^{-*}P_WA^*.
\end{equation}
Moreover, if $c,d > 0$ are such that $c\|x\|\le\|Ax\|\le d\|x\|$ for $x\in\calH$ then
\begin{equation}\label{e:norm_RW-1}
d^{-1}\min\{1,\la^{-1}cd\}\|x\|\,\le\,\|R(\la)^{-1}x\|\,\le\,c^{-1}\max\{1,\la^{-1}cd\}\|x\|.
%\beta^{-1}\|x\|\,\le\,\|R^{-1}x\|\,\le\,\alpha^{-1}\|x\|.
\end{equation}
As a consequence, we obtain
\begin{equation}\label{e:PAW}
d^{-1}\|P_WA^*x\|\,\le\,\|P_{AW}x\|\,\le\,c^{-1}\|P_WA^*x\|.
\end{equation}
\end{lem}
\begin{proof}
First of all, we note that $(AW)^\perp = A^{-*}W^\perp$. From this, it immediately follows that $R(\la)$ is boundedly invertible and that $P_{AW}R(\la) = AP_W$. The latter implies $P_{AW} = AP_WR(\la)^{-1}$. Adjoining this gives \eqref{e:PS-1W}. For the proof of \eqref{e:norm_RW-1} let $x\in\calH$. Then we obtain
\begin{align*}
\|R(\la)x\|^2
&= \|AP_Wx\|^2 + \la^2\|A^{-*}P_{W^\perp}x\|^2\ge c^2\|P_Wx\|^2 + \la^2d^{-2}\|P_{W^\perp}x\|^2\\
&\ge\min\{c^2,\la^2d^{-2}\}\|x\|^2,
\end{align*}
as well as
\begin{align*}
\|R(\la)x\|^2
&= \|AP_Wx\|^2 + \la^2\|A^{-*}P_{W^\perp}x\|^2\le d^2\|P_Wx\|^2 + \la^2c^{-2}\|P_{W^\perp}x\|^2\\
&\le\max\{d^2,\la^2c^{-2}\}\|x\|^2.
\end{align*}
This implies \eqref{e:norm_RW-1}. Setting $\la = cd$ and using \eqref{e:PS-1W} yields \eqref{e:PAW}.
\end{proof}

We briefly remark that Lemma \ref{l:R_W} immediately implies the following corollary which was already proven in \cite[Theorem 2.4]{g}.

\begin{cor}\label{c:gavruta_new}
Let $((W_i,c_i))_{i\in I}$ be a fusion frame for $\calH$ with bounds $\alpha\le\beta$ and let $A\in\calB(\calH)$ be boundedly invertible. Then also $((AW_i,c_i))_{i\in I}$ is a fusion frame for $\calH$ with bounds $\alpha\gamma^{-2}\le\beta\gamma^2$, where $\gamma = \|A\|\|A^{-1}\|$.
\end{cor}

We are now ready to prove our main result in this section.

\begin{thm}\label{thm:effect_pert_FF}
Let $\calW = ((W_i,c_i))_{i\in I}$ be a fusion frame for $\calH$ with fusion frame bounds $\alpha\le\beta$ and let $\calV = ((V_i,d_i))_{i\in I}$ be a $\mu$-perturbation of $\calW$, where $0<\mu<\sqrt{\alpha}$. If both sequences $(c_i)_{i\in I}$ and $(d_i)_{i\in I}$ are bounded from below by some $\tau > 0$ then the canonical FF-dual $\wt\calV$ of $\calV$, is a $C\mu$-perturbation of the canonical FF-dual $\wt\calW$ of $\calW$, where
$$
C = \frac{c^2+d^2}\alpha\left[ \frac{1 + (\alpha^{-1} + \beta)^2}{\sqrt{\alpha}}\left(\frac{\sqrt 2}{\tau} + cd^2\right) + d^2\left(1 + c^2d^2\right) \right]
$$
with $c := 2\sqrt\beta + \mu$ and $d := (\sqrt\alpha - \mu)^{-1}$.
\end{thm}
\begin{proof}
For $i\in I$ we define the operators
$$
R_{W_i} := S_\calW^{-1}P_{W_i} + S_\calW P_{W_i^\perp}
\qquad\text{and}\qquad
R_{V_i} := S_\calV^{-1}P_{V_i} + S_\calV P_{V_i^\perp}.
$$
By Lemma \ref{l:R_W}, these are boundedly invertible and we have
$$
P_{S_\calW^{-1}W_i} = R_{W_i}^{-*}P_{W_i}S_\calW^{-1}
\qquad\text{and}\qquad
P_{S_\calV^{-1}V_i} = R_{V_i}^{-*}P_{V_i}S_\calV^{-1}.
$$
Furthermore, 
$$
\|R_{W_i}^{-1}\|\le\max\{\beta,\alpha^{-1}\}\le\alpha^{-1}+\beta
\quad\text{and}\quad
\|R_{V_i}^{-1}\|\le\max\{c^2,d^2\}\le c^2+d^2.
$$
Put $\hat c_i := \|S_\calW^{-1}|W_i\|c_i$ and $\hat d_i := \|S_\calV^{-1}|V_i\|d_i$. Then $\hat c_i\le \alpha^{-1}c_i$ and $\hat d_i\le (\sqrt\alpha - \mu)^{-2}d_i = d^2d_i$. For $x\in\calH$ define
$$
\Delta_i(x) := \Big\|\hat c_iP_{S_\calW^{-1}W_i}x - \hat d_iP_{S_\calV^{-1}V_i}x\Big\| = \left\|\hat c_iR_{W_i}^{-*}P_{W_i}S_\calW^{-1}x - \hat d_iR_{V_i}^{-*}P_{V_i}S_\calV^{-1}x\right\|.
$$
Since $\|T_{\wt\calW}x - T_{\wt\calV}x\|^2 = \sum_{i\in I}\Delta_i^2(x)$, it is our aim to estimate $\Delta_i(x)$. We have
\begin{align*}
\Delta_i(x)
&\le \left\|\hat c_i(R_{W_i}^{-*} - R_{V_i}^{-*})P_{W_i}S_\calW^{-1}x\right\| + \left\|R_{V_i}^{-*}\left(\hat c_i P_{W_i}S_\calW^{-1}x - \hat d_iP_{V_i}S_\calV^{-1}x\right)\right\|\\
&\le \left\|R_{W_i}^{-1} - R_{V_i}^{-1}\right\|\alpha^{-1}c_i\|P_{W_i}S_\calW^{-1}x\| + \left\|R_{V_i}^{-1}\right\|\left\|\hat c_iP_{W_i}S_\calW^{-1}x - \hat d_iP_{V_i}S_\calW^{-1}x\right\|\\
& \hspace{5.72cm} + \left\|R_{V_i}^{-1}\right\|d^2d_i\left\|P_{V_i}\left(S_\calW^{-1} - S_\calV^{-1}\right)x\right\|.
\end{align*}
Since $R_{W_i}^{-1} - R_{V_i}^{-1} = R_{W_i}^{-1}(R_{V_i} - R_{W_i})R_{V_i}^{-1}$ and $\|R_{V_i}^{-1}\|\le c^2+d^2$ by Lemma \ref{l:R_W}, with
$$
\Delta_i^{(1)} := \alpha^{-1}\left\|R_{W_i}^{-1}\right\|\left\|R_{W_i} - R_{V_i}\right\|
\qquad\text{and}\qquad
\Delta_i^{(2)} := \left|\left\|S_\calW^{-1}|W_i\right\| - \left\|S_\calV^{-1}|V_i\right\|\right|
$$
we obtain
\begin{align*}
\frac{\Delta_i(x)}{c^2+d^2}
&\le \Delta_i^{(1)}c_i\left\|P_{W_i}S_\calW^{-1}x\right\| + d^2d_i\left\|P_{V_i}\left(S_\calW^{-1} - S_\calV^{-1}\right)x\right\| + \Delta_i^{(2)}c_i\left\|P_{W_i}S_\calW^{-1}x\right\|\\ 
&\hspace{3.54cm}+ \left\|S_\calV^{-1}|V_i\right\|\left\|\left(c_iP_{W_i} - d_iP_{V_i}\right)S_\calW^{-1}x\right\|\\
&\le \left(\Delta_i^{(1)} + \Delta_i^{(2)}\right)c_i\left\|P_{W_i}S_\calW^{-1}x\right\| + d^2d_i\left\|P_{V_i}\left(S_\calW^{-1} - S_\calV^{-1}\right)x\right\|\\
&\hspace{5.25cm}+d^2\left\|\left(c_iP_{W_i} - d_iP_{V_i}\right)S_\calW^{-1}x\right\|.
\end{align*}
Let us start with estimating $\Delta_i^{(2)}$. For this, observe that $S_\calW^{-1} - S_\calV^{-1} = S_\calV^{-1}(S_\calV - S_\calW)S_\calW^{-1}$ and $S_\calW - S_\calV = T_\calW^*(T_\calW - T_\calV) + (T_\calW^* - T_\calV^*)T_\calV$. Hence,
$$
\left\|S_\calW^{-1} - S_\calV^{-1}\right\|\le\frac {2\sqrt\beta + \mu}{\alpha(\sqrt\alpha - \mu)^2}\mu.
$$
From this, $\|S_\calW^{-1}|W_i\| = \|S_\calW^{-1}P_{W_i}\|$, Lemma \ref{l:PQ}, and Remark \ref{r:fusi} it thus follows that
\begin{align*}
\Delta_i^{(2)}
&= \left|\left\|S_\calW^{-1}P_{W_i}\right\| - \left\|S_\calV^{-1}P_{V_i}\right\|\right|\,\le\,\left\|S_\calW^{-1}P_{W_i} - S_\calV^{-1}P_{V_i}\right\|\\
&\le \left\|S_\calW^{-1}\left(P_{W_i} - P_{V_i}\right)\right\| + \left\|S_\calW^{-1} - S_\calV^{-1}\right\|\\
&\le \alpha^{-1}\sqrt{\frac 1 {c_i^2} + \frac 1 {d_i^2}}\;\left\|c_iP_{W_i} - d_iP_{V_i}\right\| + \frac{(2\sqrt\beta + \mu)\mu}{\alpha(\sqrt\alpha - \mu)^2}\\
&\le \frac{\sqrt 2 \mu}{\tau\alpha} + \frac{(2\sqrt\beta + \mu)\mu}{\alpha(\sqrt\alpha - \mu)^2} = M\mu,
\end{align*}
where
$$
M = \frac{1}{\alpha}\left(\frac{\sqrt 2}\tau + \frac{2\sqrt\beta + \mu}{(\sqrt\alpha - \mu)^2} \right) = \frac{1}{\alpha}\left(\frac{\sqrt 2}\tau + cd^2 \right).
$$
In order to estimate $\Delta_i^{(1)}$, we observe that
\begin{align*}
\|R_{W_i} - R_{V_i}\|
&= \left\|S_\calW^{-1}P_{W_i} + S_\calW P_{W_i^\perp} - S_\calV^{-1}P_{V_i} - S_\calV P_{V_i^\perp}\right\|\\
&\le M\mu + \left\|S_\calW P_{W_i^\perp} - S_\calV P_{V_i^\perp}\right\|\\
&\le M\mu + \left\|S_\calW\left(P_{W_i^\perp} - P_{V_i^\perp}\right)\right\| + \left\|\left(S_\calW - S_\calV\right)P_{V_i^\perp}\right\|\\
&\le \left( \frac{1}{\alpha}\left(\frac{\sqrt 2}\tau + cd^2 \right) + \frac{\sqrt 2\beta}\tau + c\right)\mu\\
&= \left(\alpha^{-1} + \beta\right)\left(\frac{\sqrt 2}{\tau} + cd^2\cdot\frac{1 + \alpha d^{-2}}{1 + \alpha\beta}\right)\mu\\
&\le \left(\alpha^{-1} + \beta\right)\left(\frac{\sqrt 2}{\tau} + cd^2\right)\mu = \alpha\left(\alpha^{-1} + \beta\right)M\mu,
\end{align*}
where in the last inequality we have used that $d^{-2}\le\beta$. Thus, we have
$$
\Delta_i^{(1)} + \Delta_i^{(2)}\,\le\,\left(\alpha^{-1} + \beta\right)^2 M\mu + M\mu = \left(1 + \left(\alpha^{-1} + \beta\right)^2\right)M\mu.
$$
Now, define the functionals
\begin{align*}
r_i(x) &:= \left(1 + \left(\alpha^{-1} + \beta\right)^2\right)M\mu c_i\left\|P_{W_i}S_\calW^{-1}x\right\|,\\
s_i(x) &:= d^2d_i\left\|P_{V_i}\left(S_\calW^{-1} - S_\calV^{-1}\right)x\right\|,\\
t_i(x) &:= d^2\left\|\left(c_iP_{W_i} - d_iP_{V_i}\right)S_\calW^{-1}x\right\|
\end{align*}
as well as
$$
R(x) := \sqrt{\sum_{i\in I}r_i^2(x)},\qquad S(x) := \sqrt{\sum_{i\in I}s_i^2(x)},
\qquad\text{and}\qquad
T(x) := \sqrt{\sum_{i\in I}t_i^2(x)}.
$$
Then $(c^2+d^2)^{-1}\Delta_i(x)\le r_i(x) + s_i(x) + t_i(x)$ and, by applying the triangle inequality on $\ell^2(I)$, one obtains
\begin{align*}
\frac{1}{(c^2+d^2)^2}\sum_{i\in I}\Delta_i^2(x)
&\le \sum_{i\in I}\left(r_i(x) + s_i(x) + t_i(x)\right)^2\,\le\,\left(R(x) + S(x) + T(x)\right)^2.
\end{align*}
We have (see Remark \ref{r:bound-for-dual})
\begin{align*}
R^2(x)&= \left(1 + \left(\alpha^{-1} + \beta\right)^2\right)^2M^2\mu^2\left\|T_\calW S_\calW^{-1}x\right\|^2\le\left(\frac{1 + \left(\alpha^{-1} + \beta\right)^2}{\sqrt\alpha}\right)^2 M^2\mu^2\|x\|^2,\\
S^2(x)&= d^4\left\|T_\calV\left(S_\calW^{-1} - S_\calV^{-1}\right)x\right\|^2
\le \alpha^{-2}c^4d^8\mu^2\|x\|^2,\\
T^2(x)&= d^4\left\|\left(T_\calW - T_\calV\right)S_\calW^{-1}x\right\|^2\le\alpha^{-2}d^4\mu^2\|x\|^2.
\end{align*}
That is,
$$
\frac{1}{c^2+d^2}\sqrt{\sum_{i\in I}\Delta_i^2(x)}\,\le\, \left[\frac{1 + \left(\alpha^{-1} + \beta\right)^2}{\sqrt\alpha}M + \alpha^{-1}c^2d^4 + \alpha^{-1}d^2\right]\mu\|x\|.
$$
This shows that $\|T_{\wt\calW} - T_{\wt\calV}\|\le C\mu$.
\end{proof}

\
\vspace{.25cm}
\noindent{\bf Acknowledgements:} The authors would like to thank Philipp Petersen (TU Berlin) for valuable discussions. Moreover, we thank S. Heineken and P. Morillas for pointing us to their paper \cite{hm}.

\end{document}